\documentclass[10pt,a4paper,reqno,english]{amsart}
\title[Scattering for 1DNLS]{Scattering for the defocusing NLS 
on the line with variable coefficients}
\def\PpP{-}
\usepackage{%
    amssymb%
    ,babel%
    ,csquotes%
    ,graphicx%
    ,mathrsfs%
    ,eucal%
    ,xcolor%
    ,enumitem%
    ,geometry%
    ,esint}
\usepackage[files,\PpP]{pagesel}%
\usepackage[%
    texencoding=ascii%
    ,backend=biber%
    ,url=false%
    ,doi=false%
    ,isbn=false%
    ,giveninits=true%
    ,maxnames=5%
    ,hyperref]{biblatex}
\usepackage[%
    pdfauthor={Piero D'Ancona}%
    ,colorlinks=true%
    ,linkcolor=magenta%
    ,citecolor=cyan%
    ,urlcolor=blue%
    ,hyperfootnotes=false%
]{hyperref}

\newcommand{\weak}{\rightharpoonup}
\newcommand{\bra}[1]{\langle #1 \rangle}
\newcommand{\one}[1]{\mathbf{1}_{#1}}

\usepackage{stmaryrd}
\SetSymbolFont{stmry}{bold}{U}{stmry}{m}{n}

\newenvironment{sproof}%
  {\baselineskip0pt\normalfont\footnotesize%
  \abovedisplayskip=1pt\abovedisplayshortskip=1pt%
  \belowdisplayskip=1pt\belowdisplayshortskip=1pt%
  \vskip2pt\par\noindent$\llbracket$\ignorespaces}%
  {\ignorespaces$\rrbracket$\vskip2pt}
\def\⟦{\begin{sproof}\ignorespaces}\def\⟧{\unskip\end{sproof}}

\numberwithin{equation}{section}
\newtheorem{theorem}{Theorem}[section]
\newtheorem{corollary}[theorem]{Corollary}
\newtheorem{lemma}[theorem]{Lemma}
\newtheorem{proposition}[theorem]{Proposition}
\theoremstyle{remark}
\newtheorem{remark}[theorem]{Remark}

\theoremstyle{definition}
\newtheorem{definition}[theorem]{Definition}

\newtheorem*{ack}{Acknowledgments}

\date{\today}

\author[P.~D'Ancona]{Piero D'Ancona}
\address{Piero D'Ancona: 
Dipartimento di Matematica\\
Sapienza Universit\`{a} di Roma\\
Piazzale A.~Moro 2\\
00185 Roma\\
Italy}
\email{dancona@mat.uniroma1.it}
\author[A.~Zanni]{Angelo Zanni}
\address{Angelo Zanni: 
Dipartimento di Matematica\\
Sapienza Universit\`{a} di Roma\\
Piazzale A.~Moro 2\\
00185 Roma\\
Italy}
\email{angelo.zanni@uniroma1.it}

\subjclass[2010]{%
Primary: 35Q55,
Secondary: 35B40,
35P25}
\keywords{%
Scattering,
nonlinear Schr\"{o}dinger equation,
variable coefficients%
}

\begin{document}

\begin{abstract}
  We prove $H^{1}$ scattering for a defocusing NLS on the line
  with fully variable coefficients. The result is proved by
  adapting the Kenig--Merle scheme
  to a non translation invariant
  setting. In addition, we give an abstract version of the scheme
  which can be applied to other operators.
\end{abstract}

\maketitle

\section{Introduction}\label{sec:1}

We consider a Schr\"{o}dinger equation on the line with
a defocusing nonlinearity
\begin{equation}\label{eq:1DNLS}
  iu_{t}-Au=|u|^{\beta-1}u,
  \qquad
  \beta>5,
\end{equation}
where $A$ is the operator with fully variable coefficients
\begin{equation}\label{eq:definA}
  Av=-\partial_{b}(a(x)\partial_{b}v)+c(x)v,
  \qquad
  \partial_{b}=\partial_{x}-ib(x).
\end{equation}
We assume that $a(x),b(x),c(x):\mathbb{R}\to \mathbb{R}$ 
are bounded functions such that
\begin{equation}\label{eq:assabc}
  \lim_{|x|\to +\infty}a(x)=1,\quad
  a(x)\ge a_{0},\quad
  c(x)\ge0,\quad
  \bra{x}^{2}(c+a_{x}^{2}+|a_{xx}|)\in L^{1}
\end{equation}
for a strictly positive constant $a_{0}$.
Under these assumptions, the linear flow $e^{-itA}$ satisfies
the usual Strichartz estimates, and global well posedness
for arbitrary $H^{1}$ data follows
(see Section \ref{sec:3}).
We recall the standard definition of scattering in the
energy space:

\begin{definition}\label{def:scattering}
  Let $u(t,x)\in C(\mathbb{R};H^{1}(\mathbb{R}))$ be a
  solution of equation \eqref{eq:definA}. We say that
  $u$ \emph{scatters} if 
  $\exists\phi_{+},\phi_{-}\in H^{1}(\mathbb{R})$
  such that $\|u(t)-e^{-itA}\phi_{\pm}\|_{H^{1}}\to0$ as
  $\pm t\to+\infty$.
\end{definition}

The first main result of the paper is that, if \eqref{eq:1DNLS}
has at least one non--scattering solution, then it has also
a nontrivial global solution whose energy
remains essentially concentrated in a bounded region of space 
for all times. More precisely we prove:

\begin{theorem}[Critical solution]\label{the:critA}
  Assume $A$ satisfies \eqref{eq:assabc}.
  If equation \eqref{eq:1DNLS} has a solution that does not 
  scatter, then it has also a global solution
  $u(t,x)\in C(\mathbb{R};H^{1}(\mathbb{R}))$ such that
  the set $\{u(t,\cdot):t\in \mathbb{R}\}$ is relatively
  compact in $H^{1}$.
\end{theorem}

The solution constructed in Theorem \ref{the:critA} is called
a \emph{critical solution}; one may think of it as
an `approximate standing wave' for \eqref{eq:1DNLS}.
The proof of Theorem \ref{the:critA} is based on the
concentration--compactness scheme introduced by
\cite{Lions84-a}, \cite{Lions84-b} and adapted
to evolution equations in
\cite{BahouriGerard97-a}, 
\cite{Keraani01-a}, 
\cite{KenigMerle06b}.
Note that the equation is not translation
invariant, and this is one of the main difficulties in adapting 
the scheme to the present setting.
Actually, we give an abstract version of the scheme,
which can be applied to other equations as well
(see Sections \ref{sec:2}--\ref{sec:4}).

The existence of a critical solution is in contrast with
the dispersive properties of the equation, thus it
is natural to expect that such a solution should not exist.
In order to prove this, we need stronger assumptions:
we assume that for some $\delta>0$
\begin{equation}\label{eq:assabcbis}
  \textstyle
  -(1-\delta)a\le x a_{x}\le (\frac 65-\delta)a,
  \qquad
  xc_{x}\le 0.
\end{equation}
Under these additional conditions, we can deduce a
virial estimate for the equation. As a consequence, we obtain
the second main result of the paper:

\begin{theorem}[Scattering]\label{the:one}
  Assume $A$ satisfies \eqref{eq:assabc}, 
  \eqref{eq:assabcbis}. Then all global solutions
  $u(t,x)\in C(\mathbb{R};H^{1}(\mathbb{R}))$ of \eqref{eq:1DNLS}
  scatter.
\end{theorem}

There exists a vast literature on scattering for NLS with
constant coefficients.
Classical results are summarized in \cite{Cazenave03-a};
in particular, the low dimensional 1D and 2D cases are harder 
and were solved in \cite{Nakanishi99-a} using smoothing
estimates with space--time weights. 
A different approach, based on interaction Morawetz estimates,
was followed in 
\cite{CollianderHolmerVisan08-a},
\cite{CollianderGrillakisTzirakis09-a}.
In \cite{KenigMerle06b} a new scheme was
introduced, based on a concentration--compactness plus
rigidity argument. The method is more abstract and flexible,
works for both focusing and defocusing equations,
and was further developed in 
\cite{DuyckaertsHolmerRoudenko08},
\cite{FangXieCazenave11-a}.

If the NLS has variable coefficients,
translational invariance breaks down, requiring
a modification of the scheme.
This was attempted
for a Schr\"{o}dinger equations with potentials in
\cite{Hong16-a},
\cite{BanicaVisciglia16-a},
\cite{Lafontaine16-a}.
Here we go a step further, proving
an abstract version of the method in order to single out the
properties of the flow which are essential for the method
to work
(see assumptions \textbf{A1--A6} in Section \ref{sec:2}
and \textbf{B1--B3} in
Sections \ref{sec:3} and \ref{sec:4}).
In particular, we recover earlier scattering
results on the line. To get the required dispersive 
propertis for the flow, we use
\cite{DAnconaFanelli06-a}.

A feature of our paper, which might be of interest to
specialists, is that only admissible Strichartz estimates
are used,
in contrast with the standard approach which involves
Strichartz estimates at non admissible points.

The plan of the paper is the following.
In Section \ref{sec:2} we prove a profile decomposition result
for a general linear flow $e^{-itA}$.
Section \ref{sec:3} builds the standard global well posedness 
theory for an abstract equation like \eqref{eq:1DNLS},
under a priori assumptions on the operator $A$; 
this includes a
scattering criterium for global solutions, and a nonlinear
perturbation lemma.
Section \ref{sec:4} is devoted to the construction and compactness
of the critical solution.
We specialize our abstract result to a Schr\"{o}dinger operator
with variable coefficients in Section \ref{sec:5}, 
completing the proof of Theorem \ref{the:critA}.
In the last Section \ref{sec:6} we prove a virial inequality 
for \eqref{eq:1DNLS} from which Theorem \ref{the:one} follows.

\begin{ack}
  The Authors are supported by the MIUR PRIN project 2020XB3EFL
 ``Hamiltonian and Dispersive PDEs''
\end{ack}

\section{Profile decomposition}\label{sec:2}

In Sections \ref{sec:2} to \ref{sec:4},
$A$ will be an arbitrary non negative selfadjoint operator on
$L^{2}(\mathbb{R})$ satisfying suitable sets of abstract
assumptions. For $z\in \mathbb{R}$, we write
\begin{equation}\label{eq:translA}
  A_{z}=\tau_{-z}A \tau_{z}.
\end{equation}
For instance, if $Au=-\Delta u+c(x)u$, then
\begin{equation*}
  A_{z}u(x)=-\Delta u(x)+c(x-z)u(x).
\end{equation*}
Fixed an index 
\begin{equation*}
  r\in(2,\infty),
\end{equation*}
we assume that for all $\psi\in H^{1}$ and all
$(x_{k})_{k\ge1},(s_{k})_{k\ge1}\subset \mathbb{R}$ 
the following holds:
\begin{description}
  \item[A1]
  $A$ is selfadjoint on $L^{2}(\mathbb{R})$ with
  $(Av,v)\ge0$ for all $v\in D(A)$,
  $A$ extends to a bounded operator $A:H^{1}\to H^{-1}$, 
  and for some $c_{0}>0$
  \begin{equation*}
    (A_{z}v,v)_{L^{2}}+c_{0}\|v\|_{L^{2}}^{2}
    \simeq\|v\|_{H^{1}}^{2}
    \quad\text{uniformly in}\  z\in \mathbb{R}.
  \end{equation*}
  \item[A2]
  $\|e^{-itA}\psi\|_{H^{1}}\lesssim\|\psi\|_{H^{1}}$
  uniformly in $t\in \mathbb{R}$.
  \item[A3] 
  The sequence $(A_{x_{k}}\psi)_{k\ge1}$ is
  precompact in $H^{-1}(\mathbb{R})$.
  \item[A4]
  If $s_{k}\to \overline{s}\in \mathbb{R}$,
  the sequence $(e^{-is_{k}A_{x_{k}}}\psi)_{k\ge1}$ is precompact
  in $H^{1}$.
  \item[A5]
  The sequence
  $(e^{-is_{k}A_{x_{k}}}\psi)_{k\ge1}$ is precompact
  in $L^{r}$.
  \item[A6]
  If $|s_{k}|\to+\infty$ then $e^{i s_{k}A_{x_{k}}}\psi\weak 0$
  in $H^{1}$ up to a subsequence.
\end{description}

In view of \textbf{A1}, we shall write
\begin{equation*}
 \|v\|_{z}^{2}:=(A_{z}v,v)_{L^{2}} =(A \tau_{z}v,\tau_{z}v).
\end{equation*}
We briefly discuss these assumptions.
\textbf{A1}--\textbf{A4} are quite natural for second order 
elliptic operators;
essentially, conditions on $A_{x_{k}}$ are implied by the
continuity of the coefficients (if $x_{k}$ is bounded) or they
are equivalent to conditions on the asymptotic behaviour 
of the coefficients at spatial infinity (if $x_{k}$ is 
unbounded).
On the other hand, \textbf{A5} and \textbf{A6} 
are more restrictive, and
embed the dispersive behaviour of the flow $e^{-itA}$
for large times.

We introduce a definition.

\begin{definition}\label{def:standseq}
  We say that $(x_{n})_{n\ge1}\subset \mathbb{R}$ is a
  \emph{standard sequence} if 
  either $x_{n}\to+\infty$, or $x_{n}\to-\infty$,
  or $x_{n}=0$ for all $n$.
\end{definition}

\begin{lemma}\label{lem:linfbound}
  Let $(v_{n})$ be a bounded sequence in $H^{1}(\mathbb{R})$.
  Then we can find a standard sequence $(x_{n})\subset \mathbb{R}$
  and $\psi\in H^{1}$ such that, up to the extraction of a
  subsequence,
  \begin{enumerate}[label=(\roman*)]
    \item $\tau_{-x_{n}}v_{n}\weak \psi$ in $H^{1}$
    \item 
    $\limsup\|v_{n}\|_{L^{\infty}}\le 
    2\|\psi\|_{L^{2}}^{1/2} \cdot \sup\|v_{n}\|_{H^{1}}^{1/2}$.
  \end{enumerate}
\end{lemma}

\begin{proof}
  Define for $R>0$ the projections 
  $P_{\le R}v=\mathcal{F}^{-1}(\one{[-R,R]}(\xi)\widehat{v}(\xi))$,
  $P_{>R}=I-P_{\le R}$.
  Recall that
  \begin{equation*}
    \textstyle
    \mathcal{F}^{-1}(\one{[-R,R]})=\frac{\sin(Rx)}{\pi x},
    \qquad
    \|\frac{\sin(Rx)}{\pi x}\|_{L^{2}}=\frac1\pi\sqrt{\frac R2}.
  \end{equation*}
  We have the easy estimate
  \begin{equation*}
    \textstyle
    \|P_{>R}v\|_{L^{\infty}}\le
    \int_{|\xi|>R}|\xi|^{-1}|\xi \widehat{v}|d\xi\le
    \sqrt{\frac 2R}\|v\|_{H^{1}}.
  \end{equation*}
  Next, pick any sequence of reals $c_{n}\downarrow1$,
  $c_{n}>1$, and let $x_{n}\in \mathbb{R}$ be such that
  $\|P_{\le R}v_{n}\|_{L^{\infty}}
    \le c_{n}|P_{\le R}v_{n}(x_{n})|$; 
  then we can write
  \begin{equation*}
    \textstyle
    \|P_{\le R}v_{n}\|_{L^{\infty}}
        \le c_{n}|P_{\le R}v_{n}(x_{n})|=
        c_{n}|\int 
        \frac{\sin(Ry)}{\pi y}v_{n}(x_{n}-y)dy|=
    c_{n}|(\frac{\sin(Rx)}{\pi x},\tau_{x_{n}}v_{n})|.
  \end{equation*}
  Combining the two estimates we get
  \begin{equation*}
    \textstyle
    \|v_{n}\|_{L^{\infty}}\le
    \sqrt{\frac 2R}\|v_{n}\|_{H^{1}}+
    c_{n}|(\frac{\sin(Rx)}{\pi x},\tau_{x_{n}}v_{n})|.
  \end{equation*}
  Passing to a subsequence we can assume that
  $\tau_{-x_{n}}v_{n}\weak \psi\in H^{1}$ and that
  $\|v_{n}\|_{L^{\infty}}\to\limsup\|v_{n}\|_{L^{\infty}}$,
  so that
  \begin{equation*}
    \textstyle
    \limsup\|v_{n}\|_{L^{\infty}}\le
    \sqrt{\frac 2R}\sup\|v_{n}\|_{H^{1}}+
    |(\frac{\sin(Rx)}{\pi x},\psi)|\le
    \sqrt{\frac 2R}\sup\|v_{n}\|_{H^{1}}+
    \frac 1\pi \sqrt{\frac R2}\|\psi\|_{L^{2}}
  \end{equation*}
  and optimizing in $R>0$
  \begin{equation*}
    \textstyle
    \le \frac{2}{\sqrt{\pi}}
    \sup\|v_{n}\|_{H^{1}}^{1/2}\|\psi\|_{L^{2}}^{1/2}.
  \end{equation*}
  Finally, by further passing to a subsequence,
  we can assume that either $x_{n}\to\pm \infty$ or 
  $x_{n}\to \overline{x}\in \mathbb{R}$.
  In the last case, we see that
  $v_{n}\weak\tau_{\overline{x}}\psi$.
  Since $\|\tau_{\overline{x}}\psi\|_{L^{2}}=\|\psi\|_{L^{2}}$,
  we can replace $\psi$ by
  $\tau_{\overline{x}}\psi$ and take $x_{n}=0$ for all $n$.
\end{proof}

\begin{proposition}[First profile]\label{pro:profiles}
  Assume the operator $A$ satisfies \textbf{A1}--\textbf{A5} 
  and suppose that a
  bounded sequence in $H^{1}(\mathbb{R})$ is given.
  Then we can extract a subsequence $(v_{n})$, 
  and we can find a function $\psi\in H^{1}$
  and standard sequences $(x_{n})$,
  $(t_{n})\subset \mathbb{R}$, such that
  \begin{enumerate}
    \item 
    $\tau_{-x_{n}}e^{-it_{n}A}v_{n}=\psi+W_{n}$
    with $W_{n}\weak 0$ in $H^{1}$
    \item 
    $\limsup\|e^{-itA}v_{n}\|_{L^{\infty}L^{\infty}}\lesssim
    \|\psi\|_{L^{2}}^{1/2}\sup\|v_{n}\|_{H^{1}}^{1/2}$
    \item 
    for $n\to+\infty$ we have
    \begin{itemize}[label=$\star$]
      \item 
      $\|v_{n}\|_{L^{2}}^{2}=\|\psi\|_{L^{2}}^{2}+
      \|W_{n}\|_{L^{2}}^{2}+o(1)$
      \item 
      $(Av_{n},v_{n})_{L^{2}}=
        \|\psi\|_{x_{n}}^{2}+
        \|W_{n}\|_{x_{n}}^{2}+o(1)$
      \item 
      $\|v_{n}\|_{L^{r}}^{r}=
        \|e^{it_{n}A_{x_{n}}}\psi\|_{L^{r}}^{r}+
        \|e^{it_{n}A_{x_{n}}}W_{n}\|_{L^{r}}^{r}+o(1)$
    \end{itemize}
  \item
  $\|\psi\|_{H^{1}}+\|W_{n}\|_{H^{1}}
    \lesssim\sup_{n}\|v_{n}\|_{H^{1}}$.
  \end{enumerate}
\end{proposition}

\begin{proof}
  Pick $t_{n}\in \mathbb{R}$ such that
  $\|e^{-it_{n}A}v_{n}\|_{L^{\infty}_{x}}\ge \frac 12
    \|e^{-itA}v_{n}\|_{L^{\infty}_{t}L^{\infty}_{x}}$.
  Clearly we can assume that either $t_{n}\to\pm \infty$, or
  $t_{n}$ converges to a finite value; we shall prove below 
  that we can actually choose $t_{n}$ to be a standard sequence.
  By \textbf{A2}, $e^{-it_{n}A}v_{n}$ is bounded in $H^{1}$.
  By Lemma \ref{lem:linfbound}, we can find a standard
  sequence $(x_{n})$ such that, up to a subsequence,
  \begin{equation*}
    \tau_{-x_{n}}e^{-it_{n}A}v_{n}=\psi+W_{n},
    \qquad
    W_{n}\weak 0 \ \text{in}\ H^{1},
  \end{equation*}
  and
  \begin{equation*}
    \limsup\|e^{-it_{n}A}v_{n}\|_{L^{\infty}}\le
    2\|\psi\|_{L^{2}}^{1/2}\sup\|e^{-it_{n}A}v_{n}\|_{H^{1}}^{1/2}
    \lesssim
    \|\psi\|_{L^{2}}^{1/2}\sup\|v_{n}\|_{H^{1}}^{1/2}.
  \end{equation*}
  This proves (1) and (2). Concerning (3) we have
  \begin{equation*}
    \|v_{n}\|_{L^{2}}^{2}=
    \|\tau_{-x_{n}}e^{-i t_{n}A}v_{n}\|_{L^{2}}^{2}=
    \|\psi+W_{n}\|_{L^{2}}^{2}=
    \|\psi\|_{L^{2}}^{2}+\|W_{n}\|_{L^{2}}^{2}+
    2\Re(\psi,W_{n})_{L^{2}}
  \end{equation*}
  and the last term is $o(1)$ since $W_{n}\weak 0$.
  For the second claim in (3) we note that
  \begin{equation*}
    (Av_{n},v_{n})
    =(A_{x_{n}}\tau_{-x_{n}}v_{n},\tau_{-x_{n}}v_{n})
    =\|\tau_{-x_{n}}v_{n}\|_{x_{n}}^{2}.
  \end{equation*}
  We can write
  \begin{equation}\label{eq:step1}
    \tau_{-x_{n}}v_{n}=
    (\tau_{-x_{n}}e^{it_{n}A}\tau_{x_{n}})
    \tau_{-x_{n}}e^{-i t_{n}A}v_{n}=
    e^{it_{n}A_{x_{n}}}\tau_{-x_{n}}e^{-i t_{n}A}v_{n}=
    e^{it_{n}A_{x_{n}}}(\psi+W_{n})
  \end{equation}
  so that
  \begin{equation*}
    \|\tau_{-x_{n}}v_{n}\|_{x_{n}}^{2}=
    \|e^{it_{n}A_{x_{n}}}\psi\|_{x_{n}}^{2}
    +
    \|e^{it_{n}A_{x_{n}}}W_{n}\|_{x_{n}}^{2}+
    2\Re(A_{x_{n}}e^{it_{n}A_{x_{n}}}\psi,
      e^{it_{n}A_{x_{n}}}W_{n})_{L^{2}}.
  \end{equation*}
  The last term coincides with
  \begin{equation*}
    2\Re(A_{x_{n}}\psi,W_{n})_{L^{2}}=o(1)
    \ \ \ \text{up to a subsequence,}
  \end{equation*}
  since by Assumption \textbf{A3} the first factor is 
  precompact in
  $H^{-1}$ while $W_{n}\weak0$ in $H^{1}$.

  To prove the last claim in (3) we use the 
  \emph{Br\'{e}zis--Lieb} theorem in the form:
  \emph{Let $f_{n},g_{n}$ be bounded sequences in $L^{r}$
  such that $g_{n}$ converges in $L^{r}$ and a.e.~while 
  $f_{n}\to0$ a.e. Then}
  \begin{equation*}
    \|f_{n}+g_{n}\|_{L^{r}}^{r}=
    \|f_{n}\|_{L^{r}}^{r}+
    \|g_{n}\|_{L^{r}}^{r}+o(1).
  \end{equation*}
  We apply this result to \eqref{eq:step1} as follows:
  \begin{equation*}
    \tau_{-x_{n}}v_{n}=f_{n}+g_{n},
    \qquad
    f_{n}=
    e^{it_{n}A_{x_{n}}}W_{n},
    \qquad
    g_{n}= e^{it_{n}A_{x_{n}}}\psi.
  \end{equation*}
  The sequence $g_{n}$ is precompact
  in $L^{r}$ by \textbf{A5} and hence up to
  a subsequence it converges in $L^{r}$.
  The sequence $f_{n}$ is bounded in $H^{1}$ hence in $L^{r}$,
  and passing to a subsequence we have $f_{n}\weak \phi$
  in $H^{1}$; by compact embedding we may also assume that
  $f_{n}\to \phi$ in $L^{2}_{loc}$ and a.e.;
  we show that $\phi=0$. Let $\chi$ be any test function;
  $(e^{-it_{n}A_{x_{n}}}\chi)$ is precompact in $L^{r}$ by 
  \textbf{A5};
  on the other hand, $W_{n}\weak 0$ in $H^{1}$, and hence
  $W_{n}\weak 0$ in $L^{r}$;
  summing up we obtain, up a subsequence,
  \begin{equation*}
    (\phi,\chi)_{L^{2}}=\lim(f_{n},\chi)=
    \lim(W_{n},e^{-it_{n}A_{x_{n}}}\chi)=
    0
  \end{equation*}
  and this implies $\phi=0$.
  Thus Br\'{e}zis--Lieb applies and we get (3).

  We now prove that if 
  $t_{n}\to \overline{t}\in \mathbb{R}$ we can modify the 
  construction so that $t_{n}=0$ for all $n$,
  i.e., we can take $t_{n}$ as a standard sequence. 
  To this end we write
  \begin{equation*}
    \tau_{-x_{n}}e^{-it_{n}A}v_{n}=
    e^{-it_{n}A_{x_{n}}}\tau_{-x_{n}} v_{n}
    \quad\text{so that}\quad 
    \tau_{-x_{n}} v_{n}=
    e^{it_{n}A_{x_{n}}} (\psi+W_{n}).
  \end{equation*}
  By precompactness in $H^{1}$, see A4, 
  passing to a subsequence we have
  \begin{equation*}
    e^{it_{n}A_{x_{n}}} \psi\to \widetilde{\psi},
    \qquad
    e^{it_{n}A_{x_{n}}}W_{n}\weak0
    \quad\text{in}\quad H^{1}.
  \end{equation*}
  Define
  \begin{equation*}
    \widetilde{W}_{n}=
    \tau_{-x_{n}} v_{n}-\widetilde{\psi}=
    e^{it_{n}A_{x_{n}}}\psi-\widetilde{\psi}
    +e^{it_{n}A_{x_{n}}}W_{n}.
  \end{equation*}
  Then we have 
  $\tau_{-x_{n}}v_{n}
    =\widetilde{\psi}+\widetilde{W}_{n}$
  with $\widetilde{W}_{n}\weak0$ in $H^{1}$;
  moreover 
  $\|\widetilde{\psi}\|_{L^{2}}=\|\psi\|_{L^{2}}+o(1)$
  and 
  $\|\widetilde{W}_{n}\|_{L^{2}}=\|W_{n}\|_{L^{2}}+o(1)$
  and similarly for the $L^{r}$ norms
  (recall $H^{1}\hookrightarrow L^{r}$)
  so that (1), (2) and the first relation in (3)
  hold with $\psi,W_{n}$ replaced by
  $\widetilde{\psi},\widetilde{W}_{n}$.
  Moreover
  \begin{equation*}
    \|e^{it_{n}A_{x_{n}}}\psi\|_{L^{r}}=
      \|\widetilde{\psi}\|_{L^{r}}+o(1),
    \qquad
    \|e^{it_{n}A_{x_{n}}}W_{n}\|_{L^{r}}=
      \|\widetilde{W}_{n}\|_{L^{r}}+o(1)
  \end{equation*}
  by construction. The remaining property in (3) follows from
  \begin{equation*}
    (A_{x_{n}}\psi,\psi)=
    (A_{x_{n}}e^{it_{n}A_{x_{n}}}\psi,e^{it_{n}A_{x_{n}}}\psi)=
    \|e^{it_{n}A_{x_{n}}}\psi\|_{x_{n}}^{2}=
    \|\widetilde{\psi}\|_{x_{n}}+o(1)
  \end{equation*}
  and the similar property for $\widetilde{W}_{n}$.

  We prove (4). Since $\tau_{-x_{n}}e^{-it_{n}A}v_{n}\weak \psi$
  in $H^{1}$, by weak semicontinuity and \textbf{A1}
  \begin{equation*}
      \|\psi\|_{H^{1}}\le
    \liminf \|\tau_{-x_{n}}e^{-it_{n}A}v_{n}\|_{H^{1}}
    \lesssim
    \liminf\|v_{n}\|_{H^{1}}\le
    \sup\|v_{n}\|_{H^{1}}.
  \end{equation*}
  The bound for $W_{n}$ then follows by the triangle
  inequality.
\end{proof}

\begin{lemma}\label{lem:conv}
  Assume \textbf{A1}--\textbf{A6}.
  Then for all sequences 
  $(x_{n}),(y_{n}),(s_{n})\subset \mathbb{R}$ 
  and $(h_{n})\subset H^{1}$ with $h_{n}\weak 0$ in $H^{1}$,
  we have:
  \begin{enumerate}
    \item
    if $|s_{n}|+|x_{n}-y_{n}|\to+\infty$ then
    $\tau_{-x_{n}}e^{is_{n}A}\tau_{y_{n}}\psi\weak0$ in $H^{1}$
    up to a subsequence
    \item 
    if $\tau_{-x_{n}}e^{is_{n}A}\tau_{y_{n}}h_{n}\weak \phi$
    in $H^{1}$ and $\phi\neq0$, then 
    $|s_{n}|+|x_{n}-y_{n}|\to+\infty$.
  \end{enumerate}
\end{lemma}

\begin{proof}
  (1): with $z_{n}=y_{n}-x_{n}$ and
  $\phi,\chi$ two test functions, write
  \begin{equation}\label{eq:quant0}
    (\tau_{-x_{n}}e^{is_{n}A}\tau_{y_{n}}\phi,\chi)_{L^{2}}=
    (\tau_{z_{n}}\phi,e^{-is_{n}A_{x_{n}}}\chi)_{L^{2}}.
  \end{equation}
  If $z_{n}$ is unbounded, passing to a subsequence we can
  assume $|z_{n}|\to+\infty$.
  Then $\tau_{z_{n}}\phi\weak0$ in $L^{r'}$,
  while $e^{-is_{n}A_{x_{n}}}\chi$ is precompact in $L^{r}$
  by \textbf{A5}, so that a subsequence
  of \eqref{eq:quant0} tends to 0
  as $n\to+\infty$. If $\psi\in H^{1}$, 
  for any $\epsilon$ we can find a test
  function $\phi$ such that $\|\phi-\psi\|_{L^{2}}<\epsilon$
  which implies
  \begin{equation*}
    |(\tau_{-x_{n}}e^{is_{n}A}\tau_{y_{n}}(\phi-\psi),
      \chi)_{L^{2}}|\le
    \epsilon\|\chi\|_{L^{2}}
  \end{equation*}
  and hence we have also, up to a subsequence,
  $(\tau_{-x_{n}}e^{is_{n}A}\tau_{y_{n}}\psi,
      \chi)_{L^{2}}\to 0$ for any test function $\chi$.
  Since $\tau_{-x_{n}}e^{is_{n}A}\tau_{y_{n}}\psi$ is bounded in
  $H^{1}$, we can extract a subsequence which converges
  weakly in $H^{1}$, and the previous argument shows that
  the limit must be 0 i.e.~claim (1) holds for unbounded $z_{n}$.

  If $z_{n}$ is bounded, we must have
  $|s_{n}|\to+\infty$. Moreover,
  the claim 
  $\tau_{-x_{n}}e^{is_{n}A}\tau_{y_{n}}\psi\weak0$ 
  in $H^{1}$
  is equivalent to
  $\tau_{-y_{n}}e^{is_{n}A}\tau_{y_{n}}\psi\weak0$,
  that is to say $e^{is_{n}A_{y_{n}}}\psi\weak0$ in $H^{1}$,
  which is true up to a subsequence by \textbf{A6}.

  (2): let $z_{n}=y_{n}-x_{n}$. 
  To prove the claim it is sufficient to prove that 
  if $|s_{n}|+|z_{n}|$ is bounded then 
  $\tau_{-x_{n}}e^{is_{n}A}\tau_{y_{n}}h_{n}\weak 0$
  in $H^{1}$; since $z_{n}$ is bounded, this is
  equivalent to
  $\tau_{-y_{n}}e^{is_{n}A}\tau_{y_{n}}h_{n}\weak 0$
  that is to say
  $e^{is_{n}A_{y_{n}}}h_{n}\weak 0$ in $H^{1}$.
  Since this is a bounded sequence in $H^{1}$, 
  it is sufficient to prove that for any test function
  $\chi$
  \begin{equation*}
    (e^{is_{n}A_{y_{n}}}h_{n},\chi)_{L^{2}}=
    (h_{n},e^{-is_{n}A_{y_{n}}}\chi)_{L^{2}}\to0.
  \end{equation*}
  Since $s_{n}$ is also bounded, we can assume that
  $s_{n}\to \overline{s}\in \mathbb{R}$, thus
  $e^{-is_{n}A_{y_{n}}}\chi$ is precompact in $H^{1}$ by A4.
  Passing to a subsequence we get the claim,
  since $h_{n}\weak0$ in $H^{1}$.
\end{proof}

\begin{theorem}[Profile decomposition]\label{the:profiledec}
  Assume \textbf{A1}--\textbf{A6}, and
  suppose that a bounded sequence in $H^{1}(\mathbb{R})$ 
  is given.
  Then we can extract a subsequence $(u_{n})_{n\ge1}$,
  and $\forall j\in \mathbb{N}$ we can find 
  $\psi_{j}\in H^{1}$ and standard sequences
  $(t^{n}_{j})_{n\ge1}$,
  $(x^{n}_{j})_{n\ge1}\subset \mathbb{R}$
  with the following properties.
  Write for $J\in \mathbb{N}$
  \begin{equation}\label{eq:profun}
    u_{n}=\sum_{j=1}^{J}
    e^{it^{n}_{j}A}\tau_{x^{n}_{j}}\psi_{j}+R^{n}_{J}.
  \end{equation}
  Then we have:
  \begin{enumerate}[label=(\alph*)]
    \item 
    $|t^{n}_{j}-t^{n}_{k}|+
      |x^{n}_{j}-x^{n}_{k}|\to+\infty$
    as $n\to+\infty$ for $j\neq k$
    \item 
    $\limsup_{n}\|e^{-itA}R^{n}_{J}\|_{L^{\infty}L^{\infty}}
      \lesssim\|\psi_{J}\|_{L^{2}}^{1/2}
      \sup_{n}\|u_{n}\|_{H^{1}}^{1/2} \to0$
    as $J\to \infty$
    \item 
    for each $J$ and $n\to+\infty$ we have
    \begin{itemize}[label=$\star$]
      \item 
      $\|u_{n}\|_{L^{2}}^{2}=\sum_{j=1}^{J}
      \|\psi_{j}\|_{L^{2}}^{2}+
      \|R^{n}_{J}\|_{L^{2}}^{2}+o(1)$
      \item 
      $(Au_{n},u_{n})_{L^{2}}=
        \sum_{j=1}^{J}
        (A \tau_{x^{n}_{j}}\psi_{j},
          \tau_{x^{n}_{j}}\psi_{j})_{L^{2}}+
        (AR^{n}_{J},R^{n}_{J})_{L^{2}}+o(1)$
      \item 
      $\|u_{n}\|_{L^{r}}^{r}=
        \sum_{j=1}^{J}
        \|e^{it_{j}^{n}A}\tau_{x^{n}_{j}}\psi_{j}\|_{L^{r}}^{r}+
        \|R^{n}_{J}\|_{L^{r}}^{r}
        +o(1)$.
    \end{itemize}
    \item 
    if $\psi_{J}=0$ for some $J$ then $\psi_{j}=0$ for all
    $j\ge J$
    \item 
    $\|\psi_{j}\|_{H^{1}}+\|R_{J}^{n}\|_{H^{1}}\lesssim
      \sup_{n}\|u_{n}\|_{H^{1}}$.
  \end{enumerate}
\end{theorem}

Note that in view of (d) we can define
\begin{equation*}
  J_{max}=
  \begin{cases}
    0 &\text{if $ \psi_{j}=0 $ for all $j\ge0$}\\
    J\in \mathbb{N} &\text{if $ \psi_{j}\neq0 $ for $j=1,\dots,J$
      and 0 afterwards}\\
    \infty &\text{if $\psi_{j}\neq0$ for all $j$.}
  \end{cases}
\end{equation*}

\begin{proof}
  The idea is to apply Proposition \ref{pro:profiles}
  iteratively. Let $(u_{n})$ be the initial sequence.
  As a first step we apply the Proposition to
  $v_{n}=u_{n}$, we extract an appropriate subsequence, 
  and we call $x^{n}_{1},t^{n}_{1}$ the
  resulting standard sequences
  and $\psi_{1}$ the resulting profile.
  Thus we can write
  \begin{equation*}
    u_{n}=e^{it^{n}_{1}A}\tau_{x^{n}_{1}}\psi_{1}+R^{n}_{1}
    \quad\text{where}\quad 
    R^{n}_{1}=e^{it^{n}_{1}A}\tau_{x^{n}_{1}}W^{n}_{1}.
  \end{equation*}
  Next, we apply the Proposition to
  $v_{n}=R^{n}_{1}$ and we get
  \begin{equation*}
    u_{n}=e^{it^{n}_{1}A}\tau_{x^{n}_{1}}\psi_{1}+
          e^{it^{n}_{2}A}\tau_{x^{n}_{2}}\psi_{2}+R^{n}_{2}
    \quad\text{where}\quad 
    R^{n}_{2}=e^{it^{n}_{2}A}\tau_{x^{n}_{2}}W^{n}_{2},
  \end{equation*}
  and so on. We finally apply a diagonal procedure to extract
  a subsequence from the original sequence; we do not change 
  notation for the extracted sequence, for the sake of simplicity.
  We next prove the claims (a)--(e).

  (c): we have 
  $\|u_{n}\|_{L^{2}}^{2}=
    \|\psi_{1}\|_{L^{2}}^{2}+\|W^{n}_{1}\|_{L^{2}}^{2}+o(1)$
  and $\|W^{n}_{1}\|_{L^{2}}=\|R^{n}_{1}\|_{L^{2}}$,
  which gives the $L^{2}$ claim for $J=1$.
  Similarly we have
  $\|R^{n}_{1}\|_{L^{2}}^{2}=
    \|\psi_{2}\|_{L^{2}}^{2}+\|W^{n}_{2}\|_{L^{2}}^{2}+o(1)$
  and $\|W^{n}_{2}\|_{L^{2}}=\|R^{n}_{2}\|_{L^{2}}$,
  proving the $L^{2}$ claim for $J=2$, and so on.

  Concerning the second claim, the first step gives
  \begin{equation*}
    (Au_{n},u_{n})_{L^{2}}=
    \|\psi_{1}\|_{x^{n}_{1}}^{2}+\|W^{n}_{1}\|_{x^{n}_{1}}^{2}
    +o(1).
  \end{equation*}
  Since
  \begin{equation*}
    \|W^{n}_{1}\|_{x^{n}_{1}}^{2}=
    (A \tau_{x^{n}_{1}}W^{n}_{1},\tau_{x^{n}_{1}}W^{n}_{1})=
    (A e^{-it^{n}_{1}A}R^{n}_{1},e^{-it^{n}_{1}A}R^{n}_{1})=
    (A R^{n}_{1},R^{n}_{1}),
  \end{equation*}
  this is the second claim for $J=1$. Iterating, we get
  the claim for all $J\ge1$.

  Concerning the $L^{r}$ claim, we have
  \begin{equation*}
    \|u_{n}\|_{L^{r}}^{r}=
    \|e^{it^{n}_{1}A_{x^{n}_{1}}}\psi_{1}\|_{L^{r}}^{r}+
    \|e^{it^{n}_{1}A_{x^{n}_{1}}}W^{n}_{1}\|_{L^{r}}^{r}
    +o(1).
  \end{equation*}
  We note that
  \begin{equation*}
    \|e^{it^{n}_{1}A_{x^{n}_{1}}}\psi_{1}\|_{L^{r}}=
    \|e^{it^{n}_{1}A}\tau_{x^{n}_{1}}\psi_{1}\|_{L^{r}},
    \qquad
    \|e^{it^{n}_{1}A_{x^{n}_{1}}}W^{n}_{1}\|_{L^{r}}=
    \|R^{n}_{1}\|_{L^{r}}
  \end{equation*}
  and this is exactly the $L^{r}$ claim for $J=1$.
  We proceed as before for $J\ge1$.

  (b): from (4) in Proposition \ref{pro:profiles} we get
  $\|R^{n}_{J}\|_{H^{1}}\lesssim\sup\|u_{n}\|_{H^{1}}<\infty$.
  Hence by (2)
  \begin{equation*}
    \limsup_{n}\|e^{-itA}R^{n}_{J}\|_{L^{\infty}L^{\infty}}
    \lesssim
      \|\psi_{J}\|_{L^{2}}^{1/2}\sup\|R^{n}_{J}\|_{H^{1}}^{1/2}
    \lesssim
    \|\psi_{J}\|_{L^{2}}^{1/2}\sup\|u_{n}\|_{H^{1}}^{1/2}.
  \end{equation*}
  By the $L^{2}$ orthogonality (c) proved before, we know that
  the series $\sum\|\psi_{j}\|_{L^{2}}^{2}$ converges.
  We conclude that $\|\psi_{J}\|_{L^{2}}\to 0$
  as $J\to+\infty$, proving (b).

  (a): we proceed by induction on $J$, i.e., we suppose the
  property holds for sequences $t^{n}_{j},x^{n}_{j}$
  with $j=1,\dots,J-1$ and we prove it for the
  sequences up to $j=J$.
  If $\psi_{J}=0$ we can actually stop the construction,
  indeed we can take $\psi_{j}=0$ for all $j\ge J$,
  $R^{n}_{j+1}=R^{n}_{j}$ for all $j\ge J$, and in
  particular we get 
  $\limsup_{n}\|R^{n}_{j}\|_{L^{\infty}L^{\infty}}=0$
  for $j\ge J$; note that this proves also (d).

  Thus we can assume $\psi_{k}$ is nonzero for all $k\le J$.
  For $\ell<J$ we can write
  \begin{equation*}
    u_{n}=
    \sum_{j=1}^{J-1}e^{it^{n}_{j}A}\tau_{x^{n}_{j}}\psi_{j}
      +R^{n}_{J-1}=
    \sum_{j=1}^{\ell}e^{it^{n}_{j}A}\tau_{x^{n}_{j}}\psi_{j}
      +R^{n}_{\ell}
  \end{equation*}
  so that
  \begin{equation*}
    R^{n}_{J-1}-R^{n}_{\ell-1}=
    -\sum_{k=\ell}^{J-1}e^{it^{n}_{k}A}\tau_{x^{n}_{k}}\psi_{k}
  \end{equation*}
  and
  \begin{equation*}
    \tau_{-x^{n}_{\ell}}e^{-it^{n}_{\ell}A}
    (R^{n}_{J-1}-R^{n}_{\ell-1})=
    -\psi_{\ell}-
    \sum_{k=\ell+1}^{J-1}
    \tau_{-x^{n}_{\ell}}
    e^{i(t^{n}_{k}-t^{n}_{\ell})A}\tau_{x^{n}_{k}}\psi_{k}.
  \end{equation*}
  Setting $s_{n}=t^{n}_{k}-t^{n}_{\ell}$,
  by the induction assumption we have
  $|s_{n}|+|x^{n}_{k}-x^{n}_{\ell}|\to+\infty$.
  Hence we can apply Lemma \ref{lem:conv}, part (1), and
  we obtain that the last sum $\weak0$ in $H^{1}$
  (up to a subsequence). On the other hand, by definition
  \begin{equation*}
    \tau_{-x^{n}_{\ell}}e^{-it^{n}_{\ell}A}R^{n}_{\ell-1}=
    \psi_{\ell}+W^{n}_{\ell}\weak \psi_{\ell}
  \end{equation*}
  at the left hand side. By cancelation we deduce
  \begin{equation*}
    h_{n}:=
    \tau_{-x^{n}_{\ell}}e^{-it^{n}_{\ell}A} R^{n}_{J-1}
    \weak0
    \quad\text{in}\quad H^{1}.
  \end{equation*}
  We compare this with
  \begin{equation*}
    \tau_{-x^{n}_{J}}e^{-it^{n}_{J}A} R^{n}_{J-1}
    \equiv
    \tau_{-x^{n}_{J}}
    e^{-it^{n}_{J}A}
    e^{it^{n}_{\ell}A}
    \tau_{x^{n}_{\ell}}h_{n}
    \weak \psi_{J}\neq0
    \quad\text{in}\quad H^{1}.
  \end{equation*}
  By Lemma \ref{lem:conv}, part (2), this implies
  $|t^{n}_{J}-t^{n}_{\ell}|+|x^{n}_{J}-x^{n}_{\ell}|\to+\infty$
  as claimed.

  Finally, (e) is a direct consequence of (4) and 
  Assumption \textbf{A1}.
\end{proof}

\section{The nonlinear theory}\label{sec:3}

Given the non negative selfadjoint operator $A$ on
$L^{2}(\mathbb{R})$, we consider now the 
Cauchy problem for the corresponding defocusing NLS
\begin{equation}\label{eq:diffNLS}
  iu_{t}-Au=|u|^{\beta-1}u,
  \qquad
  u(0,x)=\phi(x),
  \qquad
  \beta>5
\end{equation}
which we regard as an integral equation
\begin{equation}\label{eq:inteq}
  \textstyle
  u=e^{-itA}\phi-i\int_{0}^{t}e^{-i(t-s)A}F(v(s,x))ds,
  \qquad
  F(z)=|z|^{\beta-1}z.
\end{equation}
We are interested in a minimal set
of `abstract' assumptions which allow to extend the
standard NLS scattering theory to the Cauchy problem
\eqref{eq:inteq}. 

To this end, fix the indices
\begin{equation}\label{eq:baseindices}
  \textstyle
  p=4\frac{\beta+1}{\beta-1},
  \qquad
  r=\beta+1.
\end{equation}
Since $\frac 2p+\frac 1r=\frac 12$,
the couple $(p,r)$ is $L^{2}$--admissible in dimension 1,
with the usual terminology.
Our assumptions for the entire Section are a standard Strichartz
estimate for the flow $e^{-itA}$, and the usual conservation
laws for the nonlinear equation:

\begin{description}
  \item[B1] (Strichartz estimate)
  With $(p,r)$ as in \eqref{eq:baseindices},
  we assume that $e^{-itA}$ satisfies
  \begin{equation}\label{eq:introStr}
    \|e^{-itA}\phi\|_{L^{p}L^{r}}
    \lesssim\|\phi\|_{L^{2}},
    \qquad
    \|e^{-itA}\phi\|_{L^{p}H^{1}_{r}\cap L^{\infty}H^{1}}
    \lesssim\|\phi\|_{H^{1}},
  \end{equation}
  \begin{equation}\label{eq:introStrin}
    \textstyle
    \|\int_{0}^{t}e^{-i(t-\tau )A}F(\tau )d\tau \|
    _{L^{p}L^{r}\cap L^{\infty}L^{2}}
    \lesssim\|F\|_{L^{p'}L^{r'}},
  \end{equation}
  \begin{equation}\label{eq:introStrin2}
    \textstyle
    \|\int_{0}^{t}e^{-i(t-\tau )A}F(\tau )d\tau \|
      _{L^{p}H^{1}_{r}\cap L^{\infty}H^{1}}
    \lesssim\|F\|_{L^{p'}H^{1}_{r'}+L^{1}H^{1}}
  \end{equation}
  and the dual estimate
  \begin{equation}\label{eq:duaintrolstr}
    \textstyle
    \|\int e^{-i\tau A}F(\tau )d\tau \|_{H^{1}}
    \lesssim\|F\|_{L^{p'}H^{1}_{r'}}.
  \end{equation}
  By translation invariance of the norms, the same 
  estimates are valid for the flows
  $e^{-itA_{y}}$ with $y\in \mathbb{R}$, with implicit
  constants independent of $y$; moreover, the estimates
  are valid on arbitrary time intervals containing 0,
  with constants independent of the interval.

\end{description}

\begin{description}
  \item[B2] (Nonlinear energy conservation):
  let $u\in C_{I} H^{1}\cap L_{I}^{\infty}H^{1}$ be a
  solution to \eqref{eq:inteq} defined on a time interval
  containing 0. Then $u$ satisfies the mass conservation
  \begin{equation}\label{eq:masscons}
    M(t)\equiv\|u(t,\cdot)\|_{L^{2}} \simeq M(0)
    \qquad
    \forall t\in I
  \end{equation}
  and the energy conservation
  \begin{equation}\label{eq:energycons}
    E(t)\equiv (Au(t),u(t))
    +\frac{2}{\beta+1}\|u(t)\|_{L^{\beta+1}}^{\beta+1}
    \simeq E(0)    \qquad
    \forall t\in I.
  \end{equation}
\end{description}

\begin{proposition}[Global existence]\label{pro:globex}
  Assume \textbf{B1--B2}.
  Let $\phi\in H^{1}(\mathbb{R})$ and $\beta>1$.
  Then there exists a unique global solution $u(\phi)$ to
  Problem \eqref{eq:inteq} in $L^{\infty}H^{1}\cap  C H^{1}$.
  The solution satisfies the energy estimates
  $M(t)\simeq M(0)$ and $E(t)\simeq E(0)$ for all times.
\end{proposition}

\begin{proof}
  Writing $F(z)=|z|^{\beta-1}z$,
  consider the nonlinear maps $\Phi,\Psi$ given by
  \begin{equation*}
    \textstyle
    \Psi(v)=e^{-itA}\phi+\Phi(v),
    \qquad
    \Phi(v)=-i\int_{0}^{t}e^{-i(t-s)A}F(u(t,s))ds.
  \end{equation*}
  By the energy estimate \eqref{eq:introStrin2} and
  the algebra property of $H^{1}(\mathbb{R})$ we have
  \begin{equation}\label{eq:enhom}
    \|\Psi(v)\|_{L^{\infty}_{I}H^{1}}
    \lesssim
    \|\phi\|_{H^{1}}+
    \||v|^{\beta-1}v\|_{L^{1}H^{1}}
    \lesssim
    \|\phi\|_{H^{1}}+
    \|v\|_{L^{\beta}H^{1}}^{\beta},
  \end{equation}
  \begin{equation*}
    \|\Psi(v_{1})-\Psi(v_{2})\|_{L^{\infty}_{I}H^{1}}
    \lesssim
    \left[\|v_{1}\|_{L^{\beta-1}_{I}H^{1}}+
      \|v_{2}\|_{L^{\beta-1}_{I}H^{1}}\right]^{\beta-1}
    \|v_{1}-v_{2}\|_{L^{\infty}_{I}H^{1}}
  \end{equation*}
  on any interval $I=[-T,T]$, and by H\"{o}lder's
  inequality this implies
  \begin{equation}\label{eq:eninhom}
    \|\Psi(v_{1})-\Psi(v_{2})\|_{L^{\infty}_{I}H^{1}}
    \lesssim
    T
    \left[\|v_{1}\|_{L^{\infty}_{I}H^{1}}+
      \|v_{2}\|_{L^{\infty}_{I}H^{1}}\right]^{\beta-1}
    \|v_{1}-v_{2}\|_{L^{\infty}_{I}H^{1}}
  \end{equation}
  Estimate \eqref{eq:enhom} implies that if
  $M \gtrsim\|\phi\|_{H^{1}}$ the map $\Psi$
  takes the closed ball 
  \begin{equation*}
    B_{M}=\{v\in L^{\infty}_{I}H^{1}:
    \|v\|_{ L^{\infty}_{I}H^{1}}\le M\}
  \end{equation*}
  of $L^{\infty}_{I}H^{1}$
  into itself; and \eqref{eq:eninhom} implies that
  there exist $T_{0}=T_{0}(\|\phi\|_{H^{1}})$ such that
  for $T\le T_{0}$ the map $\Psi$ is a contraction
  on $B_{M}$. Thus $\Psi$ has a unique fixed point in
  $B_{M}$, which is the unique local solution to 
  Problem \eqref{eq:inteq} on the interval $[-T_{0},T_{0}]$
  Since the image of $\Psi$ is contained in
  $C_{I}H^{1}$ we obtain that the solution is also continuous
  in time with values in $H^{1}$.

  Next, extend $u$ to a maximal solution
  on $(-T^{*},T^{*})$; in virtue of 
  \eqref{eq:masscons}--\eqref{eq:energycons} we see that the
  $H^{1}$ norm of $u(t)$ for $t\in(-T^{*},T^{*})$ remains 
  bounded by a constant which depends only on $\|\phi\|_{H^{1}}$.
  This is sufficient to apply a continuation argument, since
  the lifespan $T_{0}$ of the local solution in $H^{1}$ 
  constructed above depends only on the $H^{1}$ norm of 
  the data, and this proves that the solution is global.
\end{proof}

Besides $(p,r)$ already defined in \eqref{eq:baseindices}, 
we fix two more indices
\begin{equation*}
  \textstyle
  p=4\frac{\beta+1}{\beta-1},
  \qquad
  r=\beta+1,
  \qquad
  a=\frac{2(\beta^{2}-1)}{\beta+3},
  \qquad
  b=\frac{2(\beta^{2}-1)}{\beta^{2}-2 \beta-7}.
\end{equation*}
We note that
$\frac 2p+\frac 1r=\frac 2a+\frac 1b=\frac 12$, so that
$(p,r)$ and $(a,b)$ are both admissible couples for the
1D Schr\"{o}dinger equation. Note also that
\begin{equation*}
  \textstyle
  p'=4 \frac{\beta+1}{3 \beta+5},
  \qquad
  r'=\frac{\beta+1}{\beta}.
\end{equation*}
Note that by interpolation of the estimates \eqref{eq:introStr},
Strichartz estimates hold also
for the couple $(a,b)$, i.e. we have also
\begin{equation*}
  \|e^{-itA}\phi\|_{L^{a}L^{b}}
  \lesssim\|\phi\|_{L^{2}},
  \qquad
  \|e^{-itA}\phi\|_{L^{a}H^{1}_{b}} \lesssim\|\phi\|_{H^{1}},
\end{equation*}
\begin{equation*}
  \textstyle
  \|\int_{0}^{t}e^{-i(t-\tau )A}F(\tau )d\tau \| _{L^{a}L^{b}}
  \lesssim\|F\|_{L^{p'}L^{r'}},
\end{equation*}
\begin{equation*}
  \textstyle
  \|\int_{0}^{t}e^{-i(t-\tau )A}F(\tau )d\tau \|
    _{L^{a}H^{1}_{b}}
  \lesssim\|F\|_{L^{p'}H^{1}_{r'}}.
\end{equation*}

We notice a few estimates that will be used repeatedly in the
following.
Since $\beta r'=r$ and $|F(u)|\lesssim|u|^{\beta}$,
by H\"{o}lder's inequality we have
(note that $p/p'=p-1$)
\begin{equation}\label{eq:holder}
  \|F(u)\|_{L^{p'}L^{r'}}\le
  \|u\|_{L^{\infty}L^{r}}^{\beta-\frac{p}{p'}}
  \|u\|_{L^{p}L^{r}}^{\frac{p}{p'}}=
  \|u\|_{L^{\infty}L^{r}}^{\beta-p+1}
  \|u\|_{L^{p}L^{r}}^{p-1}  
\end{equation}
Similarly, 
since $|\partial_{x}F(u)|\lesssim|u_{x}||u|^{\beta-1}$,
we have
\begin{equation}\label{eq:holderb}
  \|F(u)\|_{L^{p'}H^{1}_{r'}}\le
  \|u\|_{L^{\infty}L^{r}}^{\beta-p+1}
  \|u\|_{L^{p}L^{r}}^{p-2}
  \|u\|_{L^{p}H^{1}_{r}}.
\end{equation}
Moreover, since
$\frac{1}{\beta p'}=\frac{\beta-1}{a}+\frac 1p$
and using the embedding $H^{1}_{b}\hookrightarrow L^{r}$ we can write
\begin{equation}\label{eq:holderc}
  \|F(v)\|_{L^{p'}_{I}H^{1}_{r'}}\lesssim
  \left\|\|v\|_{L^{r}}^{\beta-1}
  \|v\|_{H^{1}_{r}}\right\|_{L^{ p'}_{I}}\le
  \|v\|_{L^{a}_{I}L^{r}}^{\beta-1}
  \|v\|_{L^{p}_{I}H^{1}_{r}}\lesssim
  \|v\|_{L^{a}_{I}H^{1}_{b}}^{\beta-1}
  \|v\|_{L^{p}_{I}H^{1}_{r}}.
\end{equation}

\begin{proposition}[Scattering]\label{pro:smallscatt}
  Assume \textbf{B1--B2}.
  Let $\phi\in H^{1}$.
  If $u(\phi)\in L^{p}L^{r}$ then $u(\phi)$ scatters, 
  that is to say, there exist $\phi_{+},\phi_{-}\in H^{1}$ 
  such that 
  \begin{equation}\label{eq:scattering}
    \|u(\phi)(t)-e^{-itA}\phi_{\pm}\|_{H^{1}}\to0
    \quad\text{as}\quad t\to\pm\infty.
  \end{equation}
  In addition, for any $M>0$ 
  $\exists\epsilon=\epsilon(M)>0$ such that if 
  $\|\phi\|_{H^{1}}<M$ and
  $\|e^{-it A}\phi\|_{L^{p}L^{r}}<\epsilon$,
  then $u(\phi)\in L^{p}L^{r}$ and scatters.
  Similarly, for any $M>0$ $\exists \delta=\delta(M)$
  such that if 
  $\|u\|_{L^{\infty}\dot H^{1}}<\delta$ and
  $\|\phi\|_{L^{2}}<M$,
  then $u(\phi)\in L^{p}L^{r}$ and scatters.
\end{proposition}

\begin{proof}
  Let $I=[0,+\infty)$, and $\Phi$ the map
  \begin{equation}\label{eq:mapphi}
    \textstyle
    \Phi(v)=-i\int_{0}^{t}e^{-i(t-s)A}F(v(s,x))ds,
    \qquad
    F(z)=|z|^{\beta-1}z
  \end{equation}
  Denote by $X_{I}$ the space
  $L^{p}_{I}H^{1}_{r}\cap L^{\infty}_{I}H^{1}$,
  with norm
  \begin{equation*}
    \|v\|_{X_{I}}=
    \|v\|_{L^{p}_{I}H^{1}_{r}\cap L^{\infty}_{I}H^{1}}.
  \end{equation*}
  By \eqref{eq:introStrin2} and
  \eqref{eq:holderb} we have, recalling that
  $H^{1}\hookrightarrow L^{r}$,
  \begin{equation}\label{eq:scatest}
    \|\Phi(v)\|_{X_{I}}\lesssim
    \|F(v)\|_{L^{p'}_{I}H^{1}_{r'}}\lesssim
    \|v\|_{L^{\infty}H^{1}}^{\beta-p+1}
    \|v\|_{L^{p}_{I}L^{r}}^{p-2}
    \|v\|_{L^{p}_{I}H^{1}_{r}}
  \end{equation}

  Now, let $u=u(\phi)$ and
  $w(t,x)=u(t+t_{0},x)$ for a fixed $t_{0}>0$. 
  Then $w$ satisfies
  \begin{equation*}
    iw_{t}=Aw+F(w),\qquad w(0)=u(t_{0})
  \end{equation*}
  that is
  \begin{equation*}
    \textstyle
    w=e^{-itA}u(t_{0})-i\int_{0}^{t}e^{-i(t-s)A}f(w)ds=
      e^{-itA}u(t_{0})+\Phi(w).
  \end{equation*}
  By \eqref{eq:scatest} we have  
  for some $c$ independent of $t_{0},\phi$
  \begin{equation*}
    \|w\|_{X_{I}}\le
    c\|e^{-itA}u(t_{0})\|_{X_{I}}+
    c\|w\|_{L^{\infty}H^{1}}^{\beta-p+1}
    \|w\|_{L^{p}_{I}L^{r}}^{p-2}
    \|w\|_{L^{p}_{I}H^{1}_{r}}
  \end{equation*}
  We have $\|w\|_{L^{\infty}H^{1}}\le C \|\phi\|_{H^{1}}$.
  Moreover, by assumption $u\in L^{p}L^{r}$, hence
  $\|u\|_{L^{p}_{[\tau,\infty)}L^{r}}\to0$ as $\tau\to+\infty$.
  Pick $t_{0}$ so large that
  \begin{equation*}
    \textstyle
    c
    (C \|\phi\|_{H^{1}})^{\beta-p+1}
    \|w\|_{L^{p}_{\mathbb{R}^{+}}L^{r}}^{p-2}
    \simeq
    \|u\|_{L^{p}_{[t_{0},\infty)}L^{r}}^{p-2}<\frac 12.
  \end{equation*}
  Absorbing one term at the right in the previous estimate,
  we get then
  \begin{equation*}
    \|w\|_{X_{I}}\le 2c\|e^{-itA}u(t_{0})\|_{X_{I}}
    \le 2c\|e^{-itA}u(t_{0})\|
    _{L^{p}H^{1}_{r}\cap L^{\infty}H^{1}}
    \lesssim\|\phi\|_{H^{1}}
  \end{equation*}
  by \eqref{eq:introStr} and the conservation laws.
  In particular we have proved that
  $u\in L^{p}_{I}H^{1}_{r}$ for $I=[t_{0},\infty)$ with $t_{0}$
  large enough. 

  Then by the dual Strichartz estimate \eqref{eq:duaintrolstr}
  and by \eqref{eq:holderb} we can write
  \begin{equation*}
    \textstyle
    \bigl\|
    \int_{T}^{+\infty}e^{isA}F(u(s))ds
    \bigr\|_{H^{1}}\lesssim
    \|F(u)\|_{L^{p'}_{[T,\infty)}H^{1}_{r'}}\lesssim
    \|u\|_{L^{\infty}H^{1}}^{\beta-p+1}
    \|u\|_{L^{p}_{[T,\infty)}L^{r}}^{p-2}
    \|u\|_{L^{p}H^{1}_{r}}
  \end{equation*}
  and we deduce that
  \begin{equation*}
    \textstyle
    \int_{T}^{+\infty}e^{isA}F(u(s))ds\to0
    \quad\text{in}\quad H^{1}
    \quad\text{as}\quad T\to +\infty.
  \end{equation*}
  This implies that the integral 
  $\int_{0}^{+\infty}e^{isA}F(u(s))ds$ converges in $H^{1}$
  and we may define
  \begin{equation*}
    \textstyle
    \phi_{+}=\phi-i\int_{0}^{+\infty}e^{isA}F(u(s))ds.
  \end{equation*}
  We obtain
  \begin{equation*}
    \textstyle
    u(t)-e^{-itA}\phi_{+}=
    -ie^{-itA}\int_{t}^{+\infty}e^{isA} F(u(s))ds\to0
  \end{equation*}
  as $t\to+\infty$, since $e^{-itA}$ is uniformly bounded
  in $H^{1}$. This proves scattering at $+\infty$,
  and the proof for negative times is the same.
 
  It remains to prove small data scattering. 
  The solution $u=u(\phi)$ satisfies
  \begin{equation*}
    u=e^{-itA}\phi+\Phi(u).
  \end{equation*}
  By \eqref{eq:holder} and \eqref{eq:introStrin}
  we get on the time interval $J=[0,T]$
  \begin{equation*}
    \|u\|_{L^{p}_{J}L^{r}}\lesssim
    \|e^{-itA}\phi\|_{L^{p}_{J}L^{r}}+
    \|u\|_{L^{\infty}L^{r}}^{\beta-p+1}
    \|u\|_{L^{p}_{J}L^{r}}^{p-1}.
  \end{equation*}
  Since
  $\|u\|_{L^{r}}\le\|u\|_{L^{2}}^{2/r}\|u\|_{L^{\infty}}^{1-2/r}
    \lesssim \|u\|_{L^{2}}^{2/r}\|u\|_{\dot H^{1}}^{1-2/r}$
  we obtain for some $\theta_{1},\theta_{2}>0$
  \begin{equation*}
    \|u\|_{L^{p}_{J}L^{r}}\lesssim
    \|e^{-itA}\phi\|_{L^{p}_{J}L^{r}}+
    \|u\|_{L^{\infty}L^{2}}^{\theta_{1}}
    \|u\|_{L^{\infty}\dot H^{1}}^{\theta_{2}}
    \|u\|_{L^{p}_{J}L^{r}}^{p-1}
  \end{equation*}
  If $\|e^{-itA}\phi\|_{L^{p}L^{r}}<\epsilon$,
  the continuous function $a(T)=\|u\|_{L^{p}_{[0,T]}L^{r}}$
  satisfies the estimate
  \begin{equation}\label{eq:continuity}
    a(T)\lesssim \epsilon+a(T)^{p-1}
  \end{equation}
  with an implicit constant depending only on
  $\|\phi\|_{H^{1}}$. If $\|\phi\|_{H^{1}}<M$
  and $\epsilon$ is sufficiently small w.r.to $M$,
  by a continuity argument we deduce from \eqref{eq:continuity}
  that $a(T)$ is bounded for all times, and hence $u$
  scatters by the first part of the proof.
  Finally, consider the last claim:
  if $\|\phi\|_{L^{2}}<M$ and
  $\|u\|_{L^{\infty}\dot H^{1}}<\delta$, 
  by the Strichartz estimate we have
  $\|e^{-itA}\phi\|_{L^{p}L^{r}}\lesssim M$ and
  the function $a(T)$ satisfies the estimate
  \begin{equation*}
    a(T)\lesssim M+\delta a(T)^{p-1}
  \end{equation*}
  again with an implicit constant depending only on $M$;
  if $\delta$ is sufficiently small w.r.to $M$,
  we conclude also in this case by a continuity argument.
\end{proof}

\begin{proposition}[Nonlinear perturbation]\label{pro:NLpert}
  Assume \textbf{B1--B2}.
  For any $M>0$ there exist $\epsilon=\epsilon(M)$
  such that the following holds.
  Let $\phi\in H^{1}$,
  $\|u\|_{L^{\infty}H^{1}}$, $\|v\|_{L^{\infty}H^{1}}$,
  $\|\rho\|_{L^{\infty}H^{1}\cap L^{p}L^{r}}$, 
  $\|\sigma\|_{L^{\infty}H^{1}\cap L^{p}L^{r}}\le M$.
  Assume $\rho-\sigma=\tau_{1}+\tau_{2}$ with
  $\|\tau_{1}\|_{L^{p}L^{r}}\le \epsilon$,
  $\|\tau_{2}\|_{L^{\infty}L^{r}}\le \epsilon$.
  Let $u$ be a solution of
  \begin{equation*}
    \textstyle
    u=e^{-itA}\phi
   -i\int_{0}^{t}
    e^{-i(t-s)A}F(u)ds
    + \rho(t,x)
  \end{equation*}
  and $v$ a solution of
  \begin{equation*}
    \textstyle
    v=e^{-itA}\phi
    -i\int_{0}^{t}
    e^{-i(t-s)A}F(v)ds
    + \sigma(t,x).
  \end{equation*}
  Then if $u(t,x)\in L^{p}L^{r}$
  we have also $v(t,x)\in L^{p}L^{r}$.
\end{proposition}

\begin{proof}
  Using again the notation \eqref{eq:mapphi},
  the functions $u,v$ satisfy
  \begin{equation*}
    u=e^{-itA}\phi+\Phi(u)+\rho,\qquad
    v=e^{-itA}\phi+\Phi(v)+\sigma
  \end{equation*}
  and the difference $h(t,x)=v-u-\sigma+\rho$ satisfies
  \begin{equation*}
    \textstyle
    h=\Phi(u+h+\sigma-\rho)-\Phi(u).
  \end{equation*}
  By \eqref{eq:introStrin} this implies,
  on a time interval $J=[0,T]$,
  \begin{equation*}
    \|h\|_{L^{p}_{J}L^{r}}
    \le
    c(p)\|F(u+h+\sigma-\rho)-F(u)\|_{L^{p'}_{J}L^{r'}}
  \end{equation*}
  Writing
  \begin{equation*}
    \textstyle
    F(u+h+\sigma-\rho)-F(u)=
    \beta\int_{0}^{1}|u+t(h+\sigma-\rho)|^{\beta-1}dt \cdot h
  \end{equation*}
  and recalling that $\rho-\sigma=\tau_{1}+\tau_{2}$,
  we get the pointwise inequality
  \begin{equation*}
    |F(u+h+\sigma-\rho)-F(u)|\lesssim_{\beta}
    |h|^{\beta}+|\tau_{1}|^{\beta}+|\tau_{2}|^{\beta}+
    |u|^{\beta-1}|h|+
    |u|^{\beta-1}|\tau_{1}|+
    |u|^{\beta-1}|\tau_{2}|.
  \end{equation*}
  Using \eqref{eq:holder} we have
  \begin{equation*}
    \||h|^{\beta}\|_{L^{p'}_{J}L^{r'}}
    \le
    \|h\|_{L^{\infty}L^{r}}^{\beta-p+1}
    \|h\|_{L^{p}_{J}L^{r}}^{p-1}
    \lesssim_{M}\|h\|_{L^{p}_{J}L^{r}}^{p-1}.
  \end{equation*}
  By a similar computation we get
  \begin{equation*}
    \||u|^{\beta-1}h\|_{L^{p'}_{J}L^{r'}}\le
    \|u\|_{L^{\infty}L^{r}}^{\beta-p+1}
    \|u\|_{L^{p}_{J}L^{r}}^{p-2}
    \|h\|_{L^{p}_{J}L^{r}}
    \lesssim_{M}
    \|u\|_{L^{p}_{J}L^{r}}^{p-2}
    \|h\|_{L^{p}_{J}L^{r}}.
  \end{equation*}
  Next we can write, as above,
  \begin{equation*}
    \||\tau_{1}|^{\beta}\|_{L^{p'}_{J}L^{r'}}
    \le
    \|\tau_{1}\|_{L^{\infty}L^{r}}^{\beta-p+1}
    \|\tau_{1}\|_{L^{p}L^{r}}^{p-1}
    \lesssim_{M}\epsilon^{p-1}
  \end{equation*}
  by the assumption $\|\tau_{1}\|_{L^{p}L^{r}}\le \epsilon$,
  while
  \begin{equation*}
    \||\tau_{2}|^{\beta}\|_{L^{p'}_{J}L^{r'}}
    \le
    \|\tau_{2}\|_{L^{\infty}L^{r}}^{\beta-p+1}
    \|\tau_{2}\|_{L^{p}L^{r}}^{p-1}
    \lesssim_{M}\epsilon^{\beta-p+1}
  \end{equation*}
  by the assumption 
  $\|\tau_{2}\|_{L^{\infty}L^{r}}\le \epsilon$.

    Finally, since $\beta p'>p$, with $\delta=\beta p'-p$,  we can show that
    (we provide the detail of this computation later)
  \begin{equation*}
	\||u|^{\beta-1}\tau_{j}\|_{L^{p'}L^{r'}}
    \le
    \|u\|_{L^{p} L^{r}}^{\frac{\beta-1}{\beta}\frac{p}{p'}}
    \|u\|_{L^{\infty} L^{r}}^{\delta \frac{\beta-1}{\beta}\frac{p}{p'}}
    \|\tau_{j}\|_{L^{p} L^{r}}^{\frac 1 \beta \frac{p}{p'}}
    \|\tau_{j}\|_{L^{\infty} L^{r}}^{\delta \frac{p}{p'}}
  \end{equation*}
  
  which implies
  \begin{equation*}
    \||u|^{\beta-1}\tau_{1}\|_{L^{p'}L^{r'}}
    \lesssim_{M}\epsilon^{\frac 1 \beta \frac{p}{p'}},
    \qquad
    \||u|^{\beta-1}\tau_{2}\|_{L^{p'}L^{r'}}
    \lesssim_{M}\epsilon^{\delta \frac{p}{p'}}.
  \end{equation*}

  Thus the continuous, nondecreasing function
  \begin{equation}\label{eq:aoft}
    a(t)=\|h\|_{L^{p}_{[0,t]}L^{r}}
  \end{equation}
  satisfies an inequality of the form
  \begin{equation*}
    a(t)\le c_{0}(M)\epsilon^{\theta}+
    c_{0}(M)
    \|u\|_{L^{p}_{[0,t]}L^{r}}^{p-2} a(t)+
    c_{0}(M)
    a(t)^{\gamma},
    \qquad
    \gamma=p-1>1
  \end{equation*}
  for a suitable $\theta>0$. Changing name to the small
  parameter $\epsilon$, we can write this
  inequality in the form
  \begin{equation}\label{eq:nlineq}
    a(t)\le \epsilon+
    c_{0}(M)
    \|u\|_{L^{p}_{[0,t]}L^{r}}^{p-2} a(t)+
    c_{0}(M)
    a(t)^{\gamma},
    \qquad
    \gamma=p-1>1
  \end{equation}
  At this point, we need a simple nonlinear lemma:

  \begin{lemma}\label{lem:NLest}
    Let $\gamma>1$, $K\ge 1$, $\eta\ge0$ and let $E$ be the set
    \begin{equation*}
      E=\{x\in \mathbb{R}:0\le x\le \eta+Kx^{\gamma}\}.
    \end{equation*}
    If 
    $\eta<\overline{\eta}(K):=
      \frac{\gamma-1}{\gamma}(K \gamma)^{\frac{1}{1-\gamma}}$
    then $E$ is made of two disjoint intervals
    $[0,x_{\gamma}]$ and $[x'_{\gamma},+\infty)$,
    with $\eta<x_{\gamma}<x'_{\gamma}$.
    We have then $x_{\gamma}=\eta+o(\eta)$ as $\eta\to0$,
    and $x_{\gamma}\le \frac{\gamma}{\gamma-1}\eta$.
  \end{lemma}

  \begin{proof}
    In the $(x,y)$--plane, consider the graph of the function
    $g(x)=\eta+Kx^{\gamma}$ and the straight line $y=x$.
    The graph of $g$ has only one point of slope 1, 
    located at $x=x_{0}=(K \gamma)^{\frac{1}{1-\gamma}}$.
    This point lies on the straight line $y=x$ if and only
    if $x_{0}=\eta+Kx_{0}^{\gamma}$; in this case,
    by convexity, the two lines are tangent to each other and
    $\eta$ must be equal to 
    $\overline{\eta}
      =\frac{\gamma-1}{\gamma}(K \gamma)^{\frac{1}{1-\gamma}}$.
    If $\eta<\overline{\eta}$ then the two lines intersect
    precisely at two points, located at $x=x_{\gamma}$ and
    $x=x'_{\gamma}$ respectively, with 
    $\eta< x_{\gamma}<x'_{\gamma}$.
    By the implicit function theorem we see that the map
    $\eta \mapsto x_{\gamma}$ is $C^{1}$, and its 
    derivative at 0 is equal to 1, implying the asymptotic 
    $x_{\gamma}=\eta+o(\eta)$ as $\eta\to0$.
    Next, consider the chord joining the points
    $(0,\eta)$ and $(x_{0},g(x_{0}))$ of the graph of $g$;
    it intersects the line $y=x$ at $(x_{1},x_{1})$ given
    by $x_{1}=\eta+Kx_{0}^{\gamma-1}x_{1}$, that is
    $x_{1}=\frac{\gamma}{\gamma-1}\eta$. Since
    the chord is above the graph of $g$, which intersects the 
    same line at $(x_{\gamma},x_{\gamma})$, we conclude that
    $x_{\gamma}<x_{1}=\frac{\gamma}{\gamma-1}\eta$.
  \end{proof}

  As a Corollary to the Lemma, we see that if $a:I\to \mathbb{R}$
  is a continuous function on an interval $I=[0,T)$ with
  $T\le \infty$, such that
  \begin{equation*}
    a(0)\le \eta,\qquad
    0\le a(t)\le \eta+K a(t)^{\gamma}
    \quad\text{for}\quad t\in I
  \end{equation*}
  with $\eta<\overline{\eta}$ as in the Lemma, then we must have
  \begin{equation*}
    a(t)\le \gamma'\eta
    \quad\text{on}\quad I,
    \qquad
    \gamma'=\frac{\gamma}{\gamma-1}.
  \end{equation*}
  Indeed, the image $a(I)$ is an interval,
  $a(I)\subseteq E$ and $a(0)\in [0,x_{\gamma}]$, hence 
  $a(I)\subseteq[0,x_{\gamma}]$.

  We apply the Lemma to the function \eqref{eq:aoft}.
  Since $u\in L^{p}L^{r}$ by assumption, we can find
  times $t_{0}=0<t_{1}<\dots t_{N}$ such that
  for $j=1,\dots,N$
  \begin{equation*}
    \textstyle
    c_{0}(M)\|u\|_{L^{p}_{[t_{j-1},t_{j}]}}^{p-2}\le \frac 12,
    \qquad
    c_{0}(M)\|u\|_{L^{p}_{[t_{N},+\infty)}}^{p-2}\le \frac 12.
  \end{equation*}
  Thus on the first interval we have
  \begin{equation*}
    \textstyle
    a(t)\le \epsilon+\frac 12a(t)+c_{0}(M)a(t)^{\gamma},
    \qquad
    t\in[0,t_{1}].
  \end{equation*}
  Absorbing one term at the left we get
  \begin{equation*}
    a(t)\le 2 \epsilon+2c_{0}(M)a(t)^{\gamma}
    \qquad
    t\in[0,t_{1}]
  \end{equation*}
  and if $2\epsilon<\overline{\eta}(2c_{0})$ i.e.~it
  is sufficiently small w.r.to $M$, we can
  apply Lemma \ref{lem:NLest} and we conclude
  \begin{equation}\label{eq:first}
    a(t)\le \gamma'\epsilon,
    \qquad t\in[0,t_{1}].
  \end{equation}
  On the second interval we have by \eqref{eq:nlineq}
  (recall $a(t)$ is nondecreasing)
  \begin{equation*}
    \textstyle
    a(t)\le \epsilon+
    \frac 12 a(t_{1})+\frac 12 a(t)+c_{0}a^{\gamma},
    \qquad
    t\in[t_{1},t_{2}]
  \end{equation*}
  and using \eqref{eq:first} we get
  \begin{equation*}
    a(t)\le (2+\gamma')\epsilon+c_{0}a^{\gamma},
    \qquad
    t\in[t_{1},t_{2}].
  \end{equation*}
  By possibly decreasing $\epsilon$ i.e.~taking
  $(2+\gamma')\epsilon<\overline{\eta}(2c_{0})$,
  we can apply the Lemma again and we obtain
  \begin{equation}\label{eq:second}
    a(t)\le \gamma'\epsilon
    \qquad\text{on}\qquad [t_{1},t_{2}].
  \end{equation}
  On the third interval we get
  \begin{equation*}
    \textstyle
    a(t)\le \epsilon+\frac 12a(t_{1})+\frac 12a(t_{2})+
    \frac 12a(t)c_{0}a^{\gamma},
    \qquad
    t\in[t_{2},t_{3}]
  \end{equation*}
  which implies
  \begin{equation*}
    \textstyle
    a(t)\le (2+2 \gamma')\epsilon+c_{0}a^{\gamma},
    \qquad
    t\in[t_{2},t_{3}]
  \end{equation*}
  We repeat the argument on each interval and in a finite
  number of steps we arrive at
  \begin{equation*}
    a(t)\le \gamma'\epsilon
    \qquad\text{on}\qquad [t_{N-1},t_{N}]
  \end{equation*}
  provided $(2+(N-1)\gamma')\epsilon<\overline{\eta}(2c_{0})$.
  One last application of the Lemma gives
  \begin{equation*}
    a(t)\le \gamma'\epsilon
    \qquad\text{on}\qquad [t_{N},+\infty)
  \end{equation*}
  and so we have proved that if $\epsilon$ is sufficiently
  small w.r.to $M$, i.e.
  $(2+N \gamma')\epsilon<\overline{\eta}(2c_{0})$, we have
  \begin{equation}\label{eq:last}
    a(t)\le \gamma'\epsilon
    \qquad\text{for all}\qquad t\ge0.
  \end{equation}
  Since $\gamma'=\frac{p-1}{p-2}<6$ this proves the claim.
\end{proof}

The last general result we need on the nonlinear problem is
the existence of the wave operator. Note that the
following Proposition applies in particular to the 
Laplace operator $A=-\partial_{x}^{2}$ as a special case:

\begin{proposition} \label{pro:waveopA}
  For any $\phi_{+}\in H^{1}$ there exists
  $\phi\in H^{1}$ such that
  the solution $u(t,x)$ of
  $iu_{t}-Au=|u|^{\beta-1}u$ with data $u(0)=\phi$ satisfies
  \begin{equation*}
    \|u(t)-e^{-itA}\phi_{+}\|_{H^{1}}\to0
    \qquad\text{as}\qquad 
    t\to+\infty,
  \end{equation*}
  and belongs to $L^{p}_{t\ge0}L^{r}$.
  A similar result holds for $t\to-\infty$.
\end{proposition}

\begin{proof}
  The proof is standard but we sketch it for the sake
  of completeness. We denote by $I$ the interval
  $I=[T,+\infty)$, with $T>0$ to be chosen,
  and we introduce the space
  \begin{equation*}
    X_{I}=L^{p}_{I}H^{1}_{r}\cap L^{a}_{I}H^{1}_{b}
  \end{equation*} 
  with its natural norm,
  and the closed ball of $X_{I}$
  \begin{equation*}
    B_{\epsilon}=\{v\in X_{I}:
    \|v\|_{X_{I}}\le \epsilon
    \}
  \end{equation*}
  with $\epsilon>0$ to be chosen.
  For $t\ge T$ and $v\in B_{\epsilon}$ we consider the map
  \begin{equation*}
    \textstyle
    \Psi(v)=
    e^{-it A}\phi_{+}+i\int_{t}^{+\infty}e^{-i(t-s)A}f(v)ds
  \end{equation*}
  for $t\ge T$. 
  Note that the Strichartz estimates
  \eqref{eq:introStrin}, \eqref{eq:duaintrolstr} are valid also
  for the integrals $\int_{t}^{+\infty}$
  (as it follows easily using \eqref{eq:duaintrolstr}).
  Thus we obtain
  \begin{equation*}
    \|\Psi(v)\|_{X_{I}}
    \lesssim
    \|e^{-itA}\phi_{+}\|_{X_{I}}+
    \|f(v)\|_{L^{p'}_{I}H^{1}_{r'}}
  \end{equation*}
  and by \eqref{eq:holderc}
  \begin{equation*}
    \|\Psi(v)\|_{X_{I}}
    \lesssim
    \|e^{-itA}\phi_{+}\|_{X_{I}}+
    \|v\|_{L^{a}H^{1}_{b}}^{\beta-1}
    \|v\|_{L^{p}H^{1}_{r}}\le
    \|e^{-itA}\phi_{+}\|_{X_{I}}+
    \epsilon^{\beta}.
  \end{equation*}
  Since $\|e^{-itA}\phi_{+}\|_{L^{p}_{I}L^{r}}\to0$
  as $T\to+\infty$, by choosing $\epsilon<1$ sufficiently
  small and $T$ sufficiently large with respect to $\epsilon$,
  we see that $\Psi:B_{\epsilon}\to B_{\epsilon}$.

  We now endow $B_{\epsilon}$ with the weaker $L^{p}_{I}L^{r}$ norm;
  note that $B_{\epsilon}$ with this norm 
  is a complete metric space (since bounded sequences in
  $L^{p}_{I}H^{1}_{r}$ and in
  $L^{a}_{I}H^{1}_{b}$ admit a weakly convergent subsequence).
  By \eqref{eq:introStrin} and H\"{o}lder's inequality
  (like in \eqref{eq:holderc}), we have
  \begin{equation*}
    \|\Psi(v_{1})-\Psi(v_{2})\|_{L^{p}_{I}L^{r}}
    \lesssim
    \|F(v_{1})-F(v_{2})\|_{L^{p'}_{I}L^{r'}}
    \lesssim
    \||v_{1}|+|v_{2}|\|_{L^{a}_{I}L^{r}}^{\beta-1}
    \|v_{1}-v_{2}\|_{L^{p}_{I}L^{r}}.
  \end{equation*}
  Since $H^{1}_{b}\hookrightarrow L^{r}$ this gives
  for all $v_{1},v_{2}\in B_{\epsilon}$
  \begin{equation*}
    \|\Psi(v_{1})-\Psi(v_{2})\|_{L^{p}_{I}L^{r}}
    \lesssim
    \||v_{1}|+|v_{2}|\|_{L^{a}_{I}H^{1}_{b}}^{\beta-1}
    \|v_{1}-v_{2}\|_{L^{p}_{I}L^{r}}
    \lesssim
    \epsilon^{\beta-1}
    \|v_{1}-v_{2}\|_{L^{p}_{I}L^{r}}.
  \end{equation*}
  By possibly taking $\epsilon$ smaller, we see that
  $\Psi:B_{\epsilon}\to B_{\epsilon}$ is a contraction
  on $B_{\epsilon}$ endowed with the $L^{p}_{I}L^{r}$
  norm, and hence has a unique fixed point $u$,
  which is a solution of the equation
  $iu_{t}-Au=f(u)$ for $t\ge T$.
  By an elementary continuation argument, like in the
  proof of global existence, we can continue this solution to
  all times $t<T$, and we obiously have
  $u\in L^{p}_{[0,+\infty)}L^{r}$.
\end{proof}

\section{The critical solution}\label{sec:4}

Given $\phi\in H^{1}$, denote by $u(\phi)$ the unique
solution of the nonlinear equation
\begin{equation*}
  iu_{t}-Au=|u|^{\beta-1}u
  \quad\text{with initial data}\quad 
  u(0)=\phi.
\end{equation*}
Recall that the energy is defined as
\begin{equation*}
  E(\phi)=
  (Au,u)_{L^{2}}+\frac{2}{\beta+1}\|u\|_{L^{\beta+1}}^{\beta+1}
\end{equation*}
We modify $E(\phi)$ by introducing the \emph{full energy},
defined as
\begin{equation*}
  \mathcal{E}(\phi)=
  E(\phi)+c_{0}\|u\|_{L^{2}}^{2}
\end{equation*}
where $c_{0}$ is the constant in \textbf{A1};
by conservation of both energy and mass,
this ensures that $\mathcal{E}(\phi)\simeq \|u(t)\|_{H^{1}}^{2}$
for all $t$.
Define the critical full energy $\mathcal{E} _{crit}$ as
\begin{equation*}
  \mathcal{E} _{crit}=
  \sup\{K>0:\forall \phi\in H^{1},
  \ \mathcal{E}(\phi)<K 
  \Rightarrow u(\phi)\in L^{p}L^{r}\}.
\end{equation*}
By Proposition \ref{pro:smallscatt} we know that 
$\mathcal{E}_{crit}>0$.
Our goal is to prove that $\mathcal{E}_{crit}=\infty$,
thus we assume that $\mathcal{E}_{crit}<\infty$ and we 
try to reach a contradiction. Since $\mathcal{E}_{crit}$ 
is finite, we can find a sequence $\phi_{n}\in H^{1}$
\begin{equation*}
  \mathcal{E}(\phi_{n})\ge \mathcal{E}_{crit},
  \quad
  \mathcal{E}(\phi_{n})\downarrow \mathcal{E}_{crit}
  \quad\text{and}\quad 
  u(\phi_{n})\not\in L^{p}L^{r}.
\end{equation*}

In order to construct a critical solution, we have to take
care of possible profiles escaping to spatial infinity.
To this end we introduce a constant coefficient,
non negative second order selfadjoint operator $A_{\infty}$
which represents the `limit operator' of $A$
as $x\to\infty$. In particular, the flow
$e^{-itA_{\infty}}$ satisfies the same Strichartz estimates
as $e^{it \Delta}$, and by a trivial modification
of standard NLS theory, the defocusing Cauchy problem
\begin{equation}\label{eq:Ainfpb}
  iu_{t}-A_{\infty}u=|u|^{\beta-1}u
  \qquad
  u(0,x)=\phi
\end{equation}
admits a unique global solution $u\in L^{p}L^{r}$, for
all initial data $\phi\in H^{1}$. 
The link between $A$ and $A_{\infty}$ is given
by the following assumption
(recall the notation \eqref{eq:translA}):

\begin{description}
  \item[B3] 
  (Asymptotic Strichartz estimates)
  We assume that the flow $e^{-itA_{\infty}}$ also satisfies
  the Strichartz estimates \eqref{eq:introStr} and
  \eqref{eq:introStrin}.
  Moreover, for any $\psi\in H^{1}$, $F\in L^{p'}L^{r'}$ and
  any sequence $(x_{n})_{n\ge1}$ with $x_{n}\to+\infty$
  or $x_{n}\to-\infty$, we have
  \begin{equation}\label{eq:asystr}
    \|(e^{-itA_{\infty}}-e^{-itA_{x_{n}}})\psi\|
        _{L^{p}L^{r}}
    \to0,
  \end{equation}
  \begin{equation}\label{eq:asystrin}
    \textstyle
    \|\int_{-\infty}^{t}
    (e^{-i(t-s)A_{\infty}}-e^{-i(t-s)A_{x_{n}}})F ds\|
      _{L^{p}L^{r}}
    \to0.
  \end{equation}
\end{description}

\begin{proposition}\label{pro:reduction}
  Assume that \textbf{A1}--\textbf{A6} and 
  \textbf{B1}--\textbf{B3} are satisfied.
  Let $(\phi_{n})$ be a bounded sequence in $H^{1}$, such that
  \begin{equation*}
    \mathcal{E}(\phi_{n})\ge \mathcal{E}_{crit},
    \quad
    \mathcal{E}(\phi_{n})\to \mathcal{E}_{crit}
    \quad\text{and}\quad 
    u(\phi_{n})\not\in L^{p}L^{r}.
  \end{equation*}
  Then $\exists\psi\in H^{1}$ such that,
  writing $\phi_{n}=\psi+R_{n}$, 
  up to a subsequence, the following holds:
  \begin{enumerate}[label=(\roman*)]
    \item 
    $\mathcal{E}(\psi)=\mathcal{E}_{crit}$ and 
    $u(\psi)\not\in L^{p}L^{r}$
    \item 
    $\|e^{-itA}R_{n}\|_{L^{\infty}L^{s}}\to0$ for all
    $s\in(2,\infty]$
    and
    $\mathcal{E}(R_{n})\to0$
    \item 
    $\phi_{n}\to \psi$ in $H^{1}$.
  \end{enumerate}
  We shall call $u(\psi)$ the \emph{critical solution}.
\end{proposition}

The proof of Proposition \ref{pro:reduction} begins by
applying Theorem \ref{the:profiledec} to the $H^{1}$ bounded
sequence $(\phi_{n})$ to obtain the profile decomposition
\begin{equation}\label{eq:profifi}
  \phi _{n}=\sum_{j=1}^{J}
  e^{it^{n}_{j}A}\tau_{x^{n}_{j}}\psi_{j}+R^{n}_{J}
\end{equation}
which satisfies properties (a)--(e) from the Theorem.
In particular from (c) we get
\begin{equation}\label{eq:sumener}
  \mathcal{E}_{crit}=
  \sum_{j=1}^{J}
  \mathcal{E}(e^{it^{n}_{j}A}\tau_{x^{n}_{j}}\psi_{j})
  +\mathcal{E}(R^{n}_{J})+o(1)
\end{equation}
and hence
\begin{equation}\label{eq:subcr}
  \infty>
  \mathcal{E}_{crit}\ge
  \limsup_{n}
  \sum_{j=1}^{J}
  \mathcal{E}(e^{it^{n}_{j}A}\tau_{x^{n}_{j}}\psi_{j}).
\end{equation}
We note also that, by interpolation, we have for all
$s\in(2,\infty]$ and a
suitable $\theta\in[0,1)$
\begin{equation}\label{eq:RLpLr}
  \|e^{-itA}R^{n}_{J}\|_{L^{\infty}L^{s}}\lesssim
  \|e^{-itA}R^{n}_{J}\|_{L^{\infty}L^{2}}^{\theta}
  \|e^{-itA}R^{n}_{J}\|_{L^{\infty}L^{\infty}}^{1-\theta}
  =o(1)
  \qquad\text{as}\quad n\to \infty.
\end{equation}

\subsection{Proof of Proposition \ref{pro:reduction}:
reduction to a unique profile} \label{sub:uniq_prof}

Our next step will be to
prove that there is exactly one profile, i.e.,
$J_{max}=1$ in the decomposition \eqref{eq:profifi}.

Note that $J_{max}>0$ since $\mathcal{E}_{crit}>0$.
Assume by contradiction that $J_{max}\ge2$.
Then from \eqref{eq:subcr} we get
\begin{equation}\label{eq:subcrit}
  \mathcal{E}(e^{it^{n}_{j}A}\tau_{x^{n}_{j}}\psi_{j})
  <\mathcal{E}_{crit}
\end{equation}
for all $j=1,\dots,J_{max}$. Indeed, if we had
$\mathcal{E}(e^{it^{n}_{j}A}\tau_{x^{n}_{j}}\psi_{j})
\ge \mathcal{E}_{crit}$
for some $j$, then for other values of $j\le J_{max}$ 
we would have
\begin{equation*}
  \mathcal{E}(e^{it^{n}_{j}A}\tau_{x^{n}_{j}}\psi_{j})\to0
  \quad \text{as}\quad n\to+\infty
\end{equation*}
implying $\|\psi_{j}\|_{L^{r}}\to0$
so that $\psi_{j}=0$, but this would contradict the 
definition of $J_{max}$.

Since all profiles satisfy \eqref{eq:subcrit},
the solutions built with such profiles as data will have
a subcritical energy in view of \eqref{eq:subcr},
provided we pick a suitable subsequence,
and hence they will scatter.
The sum of these scattering solutions will be close to
$u(\phi_{n})$ for large $n$, and this will imply that
$u(\phi_{n})$ scatters too, against its definition.
We proceed with this construction.

To each profile $\psi_{j}$ we associate a solution
$U_{j}(t,x)$ and then we define
\begin{equation}\label{eq:Ujn}
  U_{j}^{n}(t,x)=U_{j}(t-t^{n}_{j},x-x^{n}_{j}).
\end{equation}
However, the definition of $U_{j}$ depends on the behaviour
of the sequences $(t^{n}_{j})_{n\ge1}$ and
$(x^{n}_{j})_{n\ge1}$. We consider four cases;
recall that these sequences are standard.
We write for simplicity
\begin{equation*}
  F(z)=|z|^{\beta-1}z,
  \qquad
  z\in \mathbb{C}.
\end{equation*}

\textsc{Case 1:
$t^{n}_{j}=x^{n}_{j}=0$ for all $n$.}
Then $U_{j}(t,x)$ is the solution of
\begin{equation*}
  i \partial_{t}u-Au=F(u),
  \qquad
  u(0,x)=\psi_{j}
\end{equation*}
given by Proposition \ref{pro:globex}.
Note that by \eqref{eq:subcrit} we have
\begin{equation*}
  U^{n}_{j}(0,x)=e^{it^{n}_{j}A}\tau_{x^{n}_{j}}\psi_{j}
  =\psi_{j},
  \qquad
  \|U^{n}_{j}\|_{L^{p}L^{r}}\lesssim\|\psi_{j}\|_{H^{1}}.
\end{equation*}
In this case we set
\begin{equation}\label{eq:rcase1}
  r^{n}_{j}(t,x):=0.
\end{equation}

\textsc{Case 2: 
$t^{n}_{j}\to \pm \infty$ and
  $x^{n}_{j}=0$ for all $n$.}
Consider for instance the case $t^{n}_{j}\to + \infty$
(the case $t^{n}_{j}\to -\infty$ is identical).
Then we define $U_{j}(t,x)$ as the solution of the
scattering problem at $-\infty$,
given by Proposition \ref{pro:waveopA},
\begin{equation*}
  i \partial_{t}u-Au=F(u),
  \qquad
  \lim_{t\to- \infty}\|U_{j}(t)-e^{-itA}\psi_{j}\|_{H^{1}}=0.
\end{equation*}
Note in particular that we have $U_{j}\in L^{p}L^{r}$.
By time translation invariance of the equation we can write
\begin{equation*}
  \textstyle
  U^{n}_{j}(t,x)=
  U_{j}(t-t^{n}_{j},x)=
  e^{-itA}U_{j}(-t^{n}_{j},x)
  +i\int_{0}^{t}e^{-i(t-s)A}F(U_{j}(s-t^{n}_{j},x))ds
\end{equation*}
\begin{equation*}
  \textstyle
  =e^{-itA }e^{it^{n}_{j}A}\psi_{j}
  +i\int_{0}^{t}e^{-i(t-s)A }F(U^{n}_{j}(s,x))ds
  +r^{n}_{j}
\end{equation*}
where
\begin{equation*}
  r^{n}_{j}:= e^{-itA }
      (U^{n}_{j}(0)-e^{it^{n}_{j}A}\psi_{j}).
\end{equation*}
By Strichartz estimates and the scattering property we have
\begin{equation}\label{eq:rcase2}
  \|r^{n}_{j}\|_{L^{p}L^{r}}\lesssim
  \|U^{n}_{j}(0)-e^{it^{n}_{j}A}\psi_{j}
    \|_{H^{1}}=
  \|U_{j}(-t^{n}_{j})- e^{it^{n}_{j}A}\psi_{j} \|_{H^{1}}
  \to0
  \quad\text{as}\quad n\to \infty.
\end{equation}

\textsc{Case 3: 
$t^{n}_{j}=0$ for all $n$ and
  $x^{n}_{j}\to \pm \infty$.}
We define $U_{j}$ as the solution of
\begin{equation*}
  iu_{t}+A_{\infty} u=F(u),
  \qquad
  u(0,x)=\psi_{j}
\end{equation*}
where $A_{\infty}$ is the constant coefficient,
non negative operator introduced in \textbf{B3}. 
Note that the properties of $A_{\infty}$ are similar to 
$\Delta$ and hence by the standard NLS theory we have
$U_{j}\in L^{p}L^{r}$.
Since $A_{\infty}$ commutes with translations we have
\begin{equation*}
  \textstyle
  U^{n}_{j}(t,x)=
  U_{j}(t,x-x^{n}_{j})=
  e^{it A_{\infty}}\tau_{-x_{n}}\psi_{j}
  +i\int_{0}^{t}e^{-i(t-s)A_{\infty}}
  F(U_{j}(s,x-x^{n}_{j}))ds
\end{equation*}
\begin{equation*}
  \textstyle
  =
  e^{-it A}\psi_{j}+
  i\int_{0}^{t}
  e^{-i(t-s)A}F(U^{n}_{j}(s,x))ds+
  r^{n}_{j}
\end{equation*}
where $r^{n}_{j}:=I+iII$,
\begin{equation*}
  \textstyle
  I=(e^{it A_{\infty}}-e^{-it A}) \tau_{-x^{n}_{j}}\psi_{j},
\end{equation*}
\begin{equation*}
  \textstyle
  II=\int_{0}^{t}
  (e^{-i(t-s)A_{\infty}}-e^{-i(t-s)A})\tau_{-x^{n}_{j}}
  F(U_{j}(s,x))ds.
\end{equation*}
Recalling \textbf{B3} we have
\begin{equation*}
  \|I\|_{L^{p}L^{r}}=
  \|\tau_{-x^{n}_{j}}I\|_{L^{p}L^{r}}\le
  \|(e^{it A_{\infty}}-e^{-it A_{x^{n}_{j}}})\psi_{j}\|
  _{L^{p}L^{r}}\to0.
\end{equation*}
By (\ref{eq:holder}) we have
\begin{equation*}
  \|F(U_{j})\|_{L^{p'}L^{r'}} \le
  \|U_{j}\|_{L^{\infty}L^{r}}^{\beta-p+1}
  \|U_{j}\|_{L^pL^r}^{p-1} < \infty
\end{equation*}
thus by \textbf{B3} we have also
\begin{equation*}
  \textstyle
  \|II\|_{L^{p}L^{r}}=
  \|\tau_{-x^{n}_{j}}II\|_{L^{p}L^{r}}\le
  \|\int_{0}^{t}
  (e^{i(t-s)A_{\infty}}-e^{-i(t-s)A_{x^{n}_{j}}})F(U_{j})ds\|
    _{L^{p}L^{r}}\to0
\end{equation*}
and we obtain also in this case
$\|r^{n}_{j}\|_{L^{p}L^{r}}\to0$ as $n\to \infty$.

\textsc{Case 4:
$t^{n}_{j}\to \pm \infty$ and
$x^{n}_{j}\to \pm \infty$.}
Again, we focus on the case $t^{n}_{j}\to+\infty$ and define
$U_{j}(t,x)$ as the solution of the
scattering problem at $-\infty$,
given by Proposition \ref{pro:waveopA} applied to $A_{\infty}$
in place of $A$:
\begin{equation*}
  i u_{t}+A_{\infty} u=F(u),
  \qquad
  \lim_{t\to- \infty}
  \|U_{j}(t)-e^{itA_{\infty}}\psi_{j}\|_{H^{1}}=0.
\end{equation*}
Again, $U_{j}\in L^{p}L^{r}$ by standard NLS theory.
We have then
\begin{equation*}
  \textstyle
  U^{n}_{j}(t,x)=
  U_{j}(t-t^{n}_{j},x-x^{n}_{j})=
  e^{itA_{\infty}}U_{j}(-t^{n}_{j},x-x^{n}_{j})
  +i\int_{0}^{t}
  e^{-i(t-s)A_{\infty}}F(U_{j}(s-t^{n}_{j},x-x^{n}_{j}))ds
\end{equation*}
\begin{equation*}
  \textstyle
  =e^{-itA }e^{it^{n}_{j}A}\tau_{-x_{n}}\psi_{j}
  +i\int_{0}^{t}e^{-i(t-s)A }F(U^{n}_{j}(s,x))ds
  +r^{n}_{j}
\end{equation*}
where $r^{n}_{j}=I+iII$,
\begin{equation*}
  I=
  e^{itA_{\infty}}\tau_{-x^{n}_{j}}U_{j}(-t^{n}_{j},x)-
  e^{-itA }e^{it^{n}_{j}A}\tau_{-x^{n}_{j}}\psi_{j},
\end{equation*}
\begin{equation*}
  \textstyle
  II=\int_{0}^{t}(e^{-i(t-s)A_{\infty}}-e^{-i(t-s)A})
  \tau_{-x^{n}_{j}}F(U_{j}(s-t^{n}_{j},x))ds.
\end{equation*}
We may rewrite $\tau_{x^{n}_{j}}I$ in the form
\begin{equation*}
  \tau_{x^{n}_{j}}I=
  e^{itA_{\infty}}
  (U_{j}(-t^{n}_{j},x)-e^{-it^{n}_{j}A_{\infty}}\psi_{j})+
  (e^{i(t-t^{n}_{j})A_{\infty}}-e^{-i(t-t^{n}_{j})A_{-x^{n}_{j}}})
  \psi_{j}
\end{equation*}
hence by Strichartz we get
\begin{equation*}
  \|I\|_{L^{p}L^{r}}\lesssim
  \|U_{j}(-t^{n}_{j})-e^{-it^{n}_{j}A_{\infty}}\psi_{j}\|_{H^{1}}
  +
  \|(e^{i(t-t^{n}_{j})A_{\infty}}-e^{-i(t-t^{n}_{j})A_{x^{n}_{j}}})
    \psi_{j}\| _{L^{p}L^{r}}
\end{equation*}
and the first term tends to 0 by the scattering property of
$U_{j}$ while the second term tends to 0 by
\textbf{B3}. On the other hand, writing
\begin{equation*}
\begin{split}
  \tau_{x^{n}_{j}}II=&
  \textstyle
  \int_{0}^{t}(e^{-i(t-s)A_{\infty}}-e^{-i(t-s)A_{-x^{n}_{j}}})
      F(U_{j}(s-t^{n}_{j},x))ds
  \\
  =& \textstyle
  \int_{-t^{n}_{j}}^{t-t^{n}_{j}}
  (e^{-i(t-s-t^{n}_{j})A_{\infty}}-
  e^{-i(t-s-t^{n}_{j})A_{-x^{n}_{j}}})
      F(U_{j}(s,x))ds
\end{split}
\end{equation*}
we get like in Case 3
\begin{equation*}
  \textstyle
  \|II\|_{L^{p}L^{r}}
  \lesssim
  \|\int_{-t^{n}_{j}}^{t}
  (e^{-i(t-s)A_{\infty}}-e^{-i(t-s)A_{-x^{n}_{j}}})F(U_{j}) ds\|
    _{L^{p}L^{r}}
\end{equation*}
which tends to 0 by \textbf{B3}, and in conclusion we get
$\|r^{n}_{j}\|_{L^{p}L^{r}}\to0$ as $n\to \infty$.

\textsc{The approximate solution}.
We have thus constructed for all $n\ge1$ and
$j=1,\dots,J$ functions $U^{n}_{j}$ which satisfy
\begin{equation*}
  \textstyle
  U^{n}_{j}(t,x)=
  e^{-i(t-t^{n}_{j})A}\tau_{-x^{n}_{j}}\psi_{j}
  +i\int_{0}^{t}e^{-i(t-s)A }F(U^{n}_{j}(s,x))ds
  +r^{n}_{j}
\end{equation*}
with
\begin{equation*}
  \lim_{n}\|r^{n}_{j}\|_{L^{p}L^{r}}=0.
\end{equation*}
More precisely we have
\begin{equation}\label{eq:Ujntrans}
  U^{n}_{j}(t,x)=U_{j}(t-t^{n}_{j},x-x^{n}_{j})
  \quad\text{with}\quad 
  U_{j}\in L^{p}L^{r}.
\end{equation}
Then recalling \eqref{eq:profifi} we see that the sum
\begin{equation*}
  W^{n}_{J}=\sum_{j=1}^{J}U^{n}_{j}
\end{equation*}
satisfies
\begin{equation}\label{eq:approxeq}
  \textstyle
  W^{n}_{J}=e^{-itA}\phi_{n}
  +i\int_{0}^{t}
  e^{-i(t-s)A}F(W^{n}_{J})ds
  + \rho^{n}_{J}
  -e^{-itA}R^{n}_{J}
\end{equation}
where
\begin{equation}\label{eq:defrho}
  \textstyle
  \rho^{n}_{J}(t,x)=
  \sum_{j=1}^{J}r^{n}_{j}
  +i\int_{0}^{t}
  e^{-i(t-s)A}
  \left[
    \sum_{j=1}^{J}F(U^{n}_{j})
    -F(\sum_{j=1}^{J}U^{n}_{j})
  \right]ds.
\end{equation}
Recall that by \eqref{eq:RLpLr} we have
\begin{equation*}
  \lim_{n}\|e^{-itA}R^{n}_{J}\|_{L^{\infty}L^{r}}=0.
\end{equation*}
We prove next that 
\begin{equation*}
  \lim_{n}\|\rho^{n}_{J}\|_{L^{p}L^{r}}=0.
\end{equation*}
The first term in \eqref{eq:defrho}
satisfies this property since 
$\lim_{n}\|r^{n}_{J}\|_{L^{p}L^{r}}=0$.

For the integral term, we first note that the
function $F(z)=|z|^{\beta-1}z$ satisfies,
for all $J\ge1$ and $z_{1},\dots,z_{J}\in \mathbb{C}$,
\begin{equation*}
  \textstyle
  |\sum_{j=1}^{J}F(z_{j})-F(\sum_{j=1}^{J}z_{j})|\lesssim
  \sum_{j\neq k}|z_{j}|^{\beta-1}|z_{k}|.
\end{equation*}
Thus by Strichartz  we get
\begin{equation*}
  \textstyle
  \|\int_{0}^{t}
  e^{-i(t-s)A}
  \left[
    \sum_{j=1}^{J}F(U^{n}_{j})
    -F(\sum_{j=1}^{J}U^{n}_{j})
  \right]ds\|_{L^{p}L^{r}}
  \lesssim\sum_{j\neq k}
  \||U^{n}_{j}|^{\beta-1}U^{n}_{k}\|_{L^{p'}L^{r'}}.
\end{equation*}

Now pick $\Phi_{j}(t,x)\in C_{c}(\mathbb{R}^{2})$ such that
$\|U_{j}-\Phi_{j}\|_{L^{p}L^{r}}<\epsilon$
and 
$\|\Phi_{j}\|_{L^{\infty}L^{r}}\lesssim 
  \|U_{j}\|_{L^{\infty}L^{r}}$
(this can be constructed e.g.~by a cutoff plus regularization),
and let 
\begin{equation*}
  \Phi^{n}_{j}(t,x)=\Phi_{j}(t-t^{n}_{j},x-x^{n}_{j}).
\end{equation*}
We can write
\begin{equation*}
  \||U^{n}_{j}-\Phi^{n}_{j}|^{\beta-1}U^{n}_{k}\|_{L^{p'}L^{r'}}
  \le
  \bigl\|\|U^{n}_{j}-\Phi^{n}_{j}\|_{L^{r}}^{\beta-1}
    \|U^{n}_{k}\|_{L^{r}}\bigr\|_{L^{p'}}.
\end{equation*}
Noticing that $\beta p'>p$, let $\delta=\beta p'-p$
and use H\"{o}lder in time as follows:
\begin{equation*}
\begin{split}
  \textstyle
  \int|f(t)|^{(\beta-1)p'}  |g(t)|^{p'}dt= &
  \textstyle
  \int|f(t)|^{(\beta-1)\frac{\beta p'}{\beta}}
  |g(t)|^{\frac{\beta p'}{\beta}}dt=
  \int(|f|^{p})^{\frac{\beta-1}{\beta}}
  (|g|^{p})^{\frac 1 \beta}
  |f|^{\frac{\beta-1}{\beta}\delta}
  |g|^{\frac \delta \beta}
  \\
  \le&
  \textstyle
  (\int|f|^{p})^{\frac{\beta-1}{\beta}}
  (\int|g|^{p})^{\frac{1}{\beta}}
  \|f\|_{L^{\infty}}^{\frac{\beta-1}{\beta}\delta}
  \|g\|_{L^{\infty}}^{\frac \delta \beta}.
\end{split}
\end{equation*}
Continuing the previous inequality we get then
\begin{equation*}
  \textstyle
  \lesssim
  \|U^{n}_{j}-\Phi^{n}_{j}\|
    _{L^{p} L^{r}}^{\frac{\beta-1}{\beta}\frac{p}{p'}}
  \|U^{n}_{k}\|_{L^{p} L^{r}}^{\frac 1 \beta \frac{p}{p'}}
  \|U^{n}_{k}\|_{L^{\infty} L^{r}}^{\delta \frac{p}{p'}}
  \lesssim\epsilon^{\theta}
\end{equation*}
for $\theta=3+\frac 5 \beta>0$.
In a similar way we have
\begin{equation*}
  \||\Phi^{n}_{j}|^{\beta-1}(U^{n}_{k}-\Phi^{n}_{k})\|
  _{L^{p'}L^{r'}}\lesssim \epsilon^{\theta'}
\end{equation*}
for a suitable $\theta'>0$.
Summing up we have proved
\begin{equation*}
  \||U^{n}_{j}|^{\beta-1}U^{n}_{k}\|_{L^{p'}L^{r'}}\lesssim
  \epsilon^{\theta}+\epsilon^{\theta'}+
  \||\Phi^{n}_{j}|^{\beta-1}\Phi^{n}_{k}\|_{L^{p'}L^{r'}}.
\end{equation*}
Since 
$|t^{n}_{j}-t^{n}_{k}|+ |x^{n}_{j}-x^{n}_{k}|\to+\infty$,
the supports of $\Phi^{n}_{j},\Phi^{n}_{k}$ are disjoint
for large $n$ and the last term vanishes.

\textsc{Conclusion}. 
By Proposition \ref{pro:NLpert} applied to
$w=W^{n}_{J}$ and $v=u(\phi_{n})$, we see that
$u(\phi_{n})\in L^{p}L^{r}$ for large $n$,
contradicting the definition of $\phi_{n}$.
We conclude $J_{max}=1$.

\subsection{Proof of Proposition \ref{pro:reduction}:
the critical solution}\label{sub:crit_solu}

We have proved that, if $E_{crit}<\infty$ and
$\phi_{n}\in H^{1}$ are such that $E(\phi_{n})\downarrow E_{crit}$
and $u(\phi_{n})\not\in L^{p}L^{r}$, up to a subsequence,
the profile decomposition of $\phi_{n}$ has exactly one profile
$\psi_{1}=\psi$. 
Writing $t_{n}=t^{n}_{j}$, $x_{n}=x^{n}_{j}$, $R_{n}=R^{n}_{j}$, 
this means that
\begin{equation}\label{eq:profifib}
  \phi _{n}= e^{it_{n}A}\tau_{x_{n}}\psi+R_{n}
\end{equation}
where $t_{n},x_{n}$ are standard sequences, and
$\|e^{-itA}R_{n}\|_{L^{\infty}L^{s}}\to0$
for all $s\in(2,\infty]$ by \eqref{eq:RLpLr}.

We now prove that we must have $t_{n}=0$, $x_{n}=0$.
If not, this would mean that the profile $\psi=\psi_{1}$
falls in one of Cases 2, 3 or 4
in the proof of Proposition \ref{pro:reduction};
then the corresponding
solution $U_{n}=U_{1}^{n}$ constructed in the proof
belongs to $L^{p}L^{r}$ automatically (even if the energy is not
subcritical) and satisfies
\begin{equation*}
  \textstyle
  U_{n}(t,x)=
  e^{-i(t-t_{n})A}\tau_{x_{n}}\psi
  +i\int_{0}^{t}e^{-i(t-s)A }F(U_{n})ds
  +r^{n}_{j},
  \qquad
  \|r_{n}\|_{L^{p}L^{r}}\to0
\end{equation*}
while $u=u(\phi_{n})$ satisfies
\begin{equation*}
  \textstyle
  u(t,x)=
  e^{-i(t-t_{n})A}\tau_{x_{n}}\psi
  +i\int_{0}^{t}e^{-i(t-s)A }F(u)ds
  +e^{-itA}R_{n},
  \qquad
  \|e^{-itA}R_{n}\|_{L^{\infty}L^{r}}\to0.
\end{equation*}
Comparing the two equations with the aid of the
Nonlinear Perturbation Proposition \ref{pro:NLpert}, we deduce
that $u(\phi_{n})\in L^{p}L^{r}$ for large $n$, which is
impossible. We conclude that the profile $\psi$ corresponds to
Case 1, so that $t_{n}=x_{n}=0$ and \eqref{eq:profifib}
simplifies to
\begin{equation}\label{eq:profifi2}
  \phi_{n}=\psi+R_{n}.
\end{equation}
From \eqref{eq:subcr} we get then 
$\mathcal{E}_{crit}\ge \mathcal{E}(\psi)$.
But we can not have $\mathcal{E}_{crit}> \mathcal{E}(\psi)$: 
this would imply
that $u(\psi)\in L^{p}L^{r}$ and hence, again by
Proposition \ref{pro:NLpert}, by \eqref{eq:profifi2} and
by $\|e^{-itA}R_{n}\|_{L^{p}L^{r}}\to0$, that
$u(\phi_{n})\in L^{p}L^{r}$, which is impossible.
Summing up, we have constructed $\psi\in H^{1}$ such that
\begin{equation}\label{eq:critsol}
  \mathcal{E}(\psi)=\mathcal{E}_{crit},
  \qquad
  u(\psi)\not\in L^{p}L^{r}.
\end{equation}
Finally, \eqref{eq:sumener} becomes
\begin{equation*}
  \mathcal{E}_{crit}=
  \mathcal{E}(\psi)
  +\mathcal{E}(R_{n})+o(1).
\end{equation*}
Since $\mathcal{E}(\psi)=\mathcal{E}_{crit}$, 
we deduce $\mathcal{E}(R_{n})\to0$,
which implies $\phi_{n}\to \psi$ in $H^{1}$
and the proof of Proposition \ref{pro:reduction} is complete.

\subsection{Compactness of the critical solution}
\label{sub:comp_rigi}

The critical solution \eqref{eq:critsol}
has a relatively compact flow:

\begin{proposition}\label{pro:compact}
  Let $u=u(\psi)$ be the critical solution from
  Proposition \ref{pro:reduction}. 
  Then $\{u(t,\cdot):t\in \mathbb{R}\}$
  is a relatively compact subset of $H^{1}$.
\end{proposition}

\begin{proof}
  It is sufficient to prove that $\{u(t,\cdot):t\in \mathbb{R}\}$
  is relatively sequentially compact. Thus, given an
  arbitrary sequence of times $(\tau_{k})\subseteq \mathbb{R}$, let
  $\widetilde{\phi}_{k}=u(t_{k},\cdot)$; then
  $(\widetilde{\phi}_{k})$ satisfies the assumptions of 
  Proposition \ref{pro:reduction}, and by property (iii) we can
  extract a subsequence which converges in $H^{1}$, proving the
  claim.
\end{proof}

\begin{corollary}\label{cor:compactcrit}
  Let $u=u(\psi)$ be the critical solution from
  Proposition \ref{pro:reduction}. 
  Then for any $\sigma\in(0,1]$ there exists $R>0$
  such that for all $t\in \mathbb{R}$ one has
  \begin{equation}\label{eq:comcr1}
    \|u(t,\cdot)\|_{L^{2}(|x|\ge R)}\le 
    \sigma \|u(t,\cdot)\|_{L^{2}(|x|\le R)}
  \end{equation}
  and
  \begin{equation}\label{eq:comcr2}
    \textstyle
    \|u_{x}(t,\cdot)\|_{L^{2}(|x|\ge R)}
    + \frac{2}{\beta+1}\|u\|_{L^{\beta+1}(|x|\ge R)}
    \le 
    \sigma \|u_{x}(t,\cdot)\|_{L^{2}(|x|\le R)}
    + \frac{2}{\beta+1}\sigma
    \|u\|_{L^{\beta+1}(|x|\le R)}.
  \end{equation}
\end{corollary}

\begin{proof}
  From Proposition \ref{pro:compact} and the continuous embedding
  $H^{1}(\mathbb{R}) \hookrightarrow X(\mathbb{R})$ where
  $X=\dot H^{1}$ or $X=L^{q}$,
  we see that the flow $\{u(t,\cdot)\}$ is compact also in
  $\dot H^{1}(\mathbb{R})$ and $L^{q}(\mathbb{R})$
  for all $q\in[2,\infty)$.

  Consider first \eqref{eq:comcr1}.
  Assume by contradiction that the claim \eqref{eq:comcr1}
  is false; this means we can find $\sigma_{0}\in(0,1]$
  and two sequences $R_{n}\uparrow \infty$ and
  $t_{n}\in \mathbb{R}$ such that for all $n$
  \begin{equation*}
    \|u(t_{n},\cdot)\|_{L^{2}(|x|\ge R_{n})}> 
    \sigma_{0} \|u(t_{n},\cdot)\|_{L^{2}(|x|\le R_{n})}.
  \end{equation*}
  By compactness, extracting a subsequence, we can assume
  that $u(t_{n})$ converges to a limit $\phi$ in $L^{2}$.
  By conservation of the $L^{2}$ norm of $u$, which is not
  null since $\mathcal{E}_{crit}>0$, we see that $\phi$ 
  is not 0. Given $0<\epsilon<\sigma_{0}\|\phi\|_{L^{2}}/4$,
  we can find $n_{\epsilon}$ such that
  $\|u(t_{n})-\phi\|_{L^{2}}<\epsilon$ for $n>n_{\epsilon}$
  and we get
  \begin{equation*}
    \|\phi\|_{L^{2}(|x|\ge R_{n})}\ge
    \|u(t_{n})\|_{L^{2}(|x|\ge R_{n})}-\epsilon\ge
    \sigma_{0} \|u(t_{n})\|_{L^{2}(|x|\le R_{n})}-\epsilon\ge
    \sigma_{0} \|\phi\|_{L^{2}(|x|\le R_{n})}-2\epsilon
  \end{equation*}
  which implies
  \begin{equation*}
    \textstyle
    \|\phi\|_{L^{2}(|x|\ge R_{n})}>
    \frac 12\sigma_{0} \|\phi\|_{L^{2}(|x|\le R_{n})}
  \end{equation*}
  which is in contradiction with $\phi\not\equiv0$,
  and this proves \eqref{eq:comcr1}. 

  A similar argument, using conservation of energy instead of
  conservation of mass, proves \eqref{eq:comcr2}.
\end{proof}

\section{Dispersive properties of the flow}\label{sec:5}

In order to prove that the flow $e^{-itA}$ satisfies
assumptions \textbf{A1--A6, B1--B3}, the essential tool will
be the boundedness of the wave operator proved in
\cite{DAnconaFanelli06-a}. We recall that given
two selfadjoint operators $H,H_{0}$ on (say) $L^{2}(\mathbb{R})$,
the \emph{wave operators} relating the corresponding
unitary groups are defined as
\begin{equation*}
  W_{\pm}(H_{0},H)f=L^{2}-\lim_{s\to\pm \infty}
  e^{isH}e^{-is H_{0}}f.
\end{equation*}
The relevant property of $W_{\pm}$ for dispersive estimates
is the \emph{intertwining property}, which in our case
takes the form
\begin{equation}\label{eq:interw}
  W_{\pm}f(H_{0})W_{\pm}^{*}=f(H)P_{ac}
\end{equation}
valid for any Borel function $f$ on $\mathbb{R}$.
Here $P_{ac}$ represents the projection of $L^{2}(\mathbb{R})$
on the continuous spectrum of $H$. 

In particular, in \cite{DAnconaFanelli06-a} the case of
\begin{equation*}
  H_{0}=-\partial^{2}_{x},
  \qquad
  H=-\partial^{2}_{x}+V(x), \quad
  V\in L^{1}(\mathbb{R})
\end{equation*}
was considered. In this case, the wave operators are well
defined and bounded from $L^{2}(\mathbb{R})$ onto the
absolutely continuous subspace of $H$.
Note that under the assumption $V\in L^{1}$
the spectrum of $H=-\partial^{2}+V$ consists only of
the absolutely continuous spectrum $[0,+\infty)$
plus a finite number of negative eigenvalues; if we can
by some additional assumption rule out the negative
eigenvalues (e.g., if the operator $H$ is non negative)
then the spectrum of $H$ is purely absolutely continuous
and we can omit $P_{ac}=I$ from formula \eqref{eq:interw}.
The following result holds:

\begin{theorem}[\cite{DAnconaFanelli06-a}] \label{the:DFW}
  Assume $\bra{x}^{2}V\in L^{1}$. Then the wave operators
  $$W_{\pm}(-\partial^{2},-\partial^{2}+V)$$ extend to bounded
  operators on $L^{p}(\mathbb{R})$, for all $p\in(1,\infty)$.
\end{theorem}

If $P_{ac}=I$, combining the previous result with
the intertwining property \eqref{eq:interw} and the
corresponding properties of the free group
$e^{it \Delta}$, we obtain immediately the dispersive
estimates
\begin{equation}\label{eq:dispV}
  \|e^{-itH}f\|_{L^{p}}\lesssim|t|^{\frac 1p-\frac 12}
  \|f\|_{L^{p'}},
  \qquad
  p\in[2,\infty).
\end{equation}
From \eqref{eq:dispV}, by standard methods one gets
the full set of Strichartz estimates
\begin{equation}\label{eq:strichV}
  \textstyle
  \|e^{-itH}f\|_{L^{a}L^{b}}\lesssim
  \|f\|_{L^{2}},
  \qquad
  \|\int_{0}^{t}e^{-i(t-s)H}F(s)ds\|_{L^{a}L^{b}}\lesssim
  \|F\|_{L^{\widetilde{a'}}L^{\widetilde{b'}}}
\end{equation}
for all admissible couples $(a,b)$ and 
$(\widetilde{a},\widetilde{b})$, satisfying the usual
condition $\frac 2a+\frac 1b=\frac 12$ and similarly for
$(\widetilde{a},\widetilde{b})$.

We shall also need Strichartz estimates for the first
derivative of the solution:
\begin{equation}\label{eq:strichVder}
  \textstyle
  \|e^{-itH}f\|_{L^{a}H^{1}_{b}}\lesssim
  \|f\|_{H^{1}},
  \qquad
  \|\int_{0}^{t}e^{-i(t-s)H}F(s)ds\|_{L^{a}H^{1}_{b}}\lesssim
  \|F\|_{L^{\widetilde{a'}}H^{1}_{\widetilde{b}'}}.
\end{equation}
To prove the first of \eqref{eq:strichVder}, we apply 
\eqref{eq:strichV} to the initial data
$(k+H)^{1/2}f$ and we get
\begin{equation}\label{eq:str1}
  \|(k+H)^{1/2}e^{itH}f\|_{L^{a}L^{b}}\lesssim
  \|(k+H)f\|_{L^{2}}\simeq_{k}\|f\|_{H^{1}}
\end{equation}
for some fixed constant $k>0$.
Next, we use the Riesz transform estimate
\begin{equation}\label{eq:ausch}
  \|\partial_{x} u\|_{L^{b}}\lesssim\|(k+H)^{1/2}f\|_{L^{b}},
  \qquad 1<b<\infty
\end{equation}
which is proved e.g.~in Theorem 1.2 of
\cite{AuscherBen-Ali07-a}. Estimate
\eqref{eq:ausch} is valid provided the potential $1+V$
in the operator $1+H=-\partial^{2}+(k+V)$ belongs to the
reverse H\"{o}lder class $RH_{b/2}$, defined by the
condition: for all intervals $I \subset \mathbb{R}$,
\begin{equation*}
  \textstyle
  (\fint_{I}(k+V)^{p/2})^{2/b}\lesssim\fint_{I} (k+V).
\end{equation*}
The reverse H\"{o}lder condition is trivially true provided
$k$ is large enough and $V$ is bounded.
Combining \eqref{eq:str1} and \eqref{eq:ausch} we obtain
the first of \eqref{eq:strichVder}.
The nonhomogeneous estimate in \eqref{eq:strichVder} is
proved in a similar way, or more directly via a standard
application of the Christ--Kiselev Lemma.

Consider now the flow $e^{-itA}$ corresponding to the operator
$A$. Applying to the equation the gauge transform
\begin{equation*}
  \textstyle
  w(t,x)=e^{-iB(x)}u(t,x),
  \qquad
  B(x)=\int_{0}^{x}b(y)dy
\end{equation*}
we obtain
\begin{equation*}
  w_{x}=e^{-iB(x)}\partial_{b}u,
  \qquad
  (aw_{x})_{x}=
  e^{-iB(x)}\partial_{b}(a \partial_{b}u),
  \qquad
  c(x)w=e^{-iB(x)}c(x)u.
\end{equation*}
Hence, if $u$ is a solution of
\begin{equation*}
  iu_{t}-Au=F,\qquad
  u(0,x)=\psi(x),
\end{equation*}
multiplying the equation by $e^{-iB(x)}$ and
working with $w=e^{-iB(x)}u$, we see that $w$ satisfies the
equation
\begin{equation*}
  iw_{t}+(a(x)\partial_{x}u)_{x}-c(x)w=e^{-iB(x)}F.
\end{equation*}
Next, define a new function $v(t,x)$ via the change of variables
\begin{equation}\label{eq:cdv}
  \textstyle
  w(t,x)=a(x)^{-\frac 14}v(t,\alpha(x)),
  \qquad
  \alpha(x)=\int_{0}^{x}\frac{ds}{a(s)^{1/2}}.
\end{equation}
If we assume $a(x)$ lies between two positive constants
$0<a_{0}\le a(x)\le a_{1}$, the new independent variable
$\alpha(x)\simeq x$. Then $v$ satisfies the equation
\begin{equation*}
  iv_{t}+v_{xx}-\widetilde{c}v=\widetilde{F},
  \qquad
  v(0,x)=\widetilde{\psi}(x)
\end{equation*}
where $\widetilde{F}$, $\widetilde{\psi}$ are determined by
\begin{equation*}
  e^{-iB(x)}F(t,x)=a^{\frac 14}\widetilde{F}(t,\alpha(x)),
  \qquad
  \psi(x)=a^{\frac 14}\widetilde{\psi}(\alpha(x)).
\end{equation*}
and the new potential $\widetilde{c}(x)$ is given by
\begin{equation*}
  \widetilde{c}(\alpha(x))=
  c(x)+\frac{a_{x}^{2}}{16a}-\frac{a_{xx}}4.
\end{equation*}
Note that since $A$ is non negative and
has a purely continuous spectrum,
the operator $-\partial^{2}+\widetilde{c}$ has the same
properties. Thus if we assunme
\begin{equation}\label{eq:assctil}
  c(x)+\frac{a_{x}^{2}}{16a}-\frac{a_{xx}}4
  \in \bra{x}^{-2} L^{1}\cap L^{\infty}
\end{equation}
we see that $v$ satisfies the dispersive estimate
\eqref{eq:dispV} and the Strichartz estimates
\eqref{eq:strichV} and \eqref{eq:strichVder}.
Coming back to $w(t,x)$ via the inverse of transformation
\eqref{eq:cdv}, we see that $w(t,x)$ satisfies the same
estimates. Finally, reintroducting the magnetic potential
$b(x)$ and noticing that for all $q$
\begin{equation*}
  \|e^{iB(x)}f\|_{H^{1}_{q}}\lesssim(1+\|b\|_{L^{\infty}}) 
  \|f\|_{H^{1}_{q}}
\end{equation*}
we obtain the following result:

\begin{proposition}\label{pro:strichA}
  Assume $a_{1}\ge a(x)\ge a_{0}>0$, $b(x)\in L^{\infty}$,
  $c(x)\ge0$ and \eqref{eq:assctil}.
  Then the flow $e^{-itA}$ satisfies the dispersive estimate
  \begin{equation}\label{eq:dispA}
    \|e^{-itA}f\|_{L^{p}}\lesssim|t|^{\frac 1p-\frac 12}
    \|f\|_{L^{p'}},
    \qquad
    p\in[2,\infty)
  \end{equation}
  and, for all admissible couples $(a,b)$ and 
  $(\widetilde{a},\widetilde{b})$, the Strichartz estimates
  \begin{equation}\label{eq:strichA}
    \textstyle
    \|e^{-itA}f\|_{L^{a}L^{b}}\lesssim
    \|f\|_{L^{2}},
    \qquad
    \|\int_{0}^{t}e^{-i(t-s)A}F(s)ds\|_{L^{a}L^{b}}\lesssim
    \|F\|_{L^{\widetilde{a'}}L^{\widetilde{b'}}}
  \end{equation}
  \begin{equation}\label{eq:strichAder}
    \textstyle
    \|e^{-itA}f\|_{L^{a}H^{1}_{b}}\lesssim
    \|f\|_{H^{1}},
    \qquad
    \|\int_{0}^{t}e^{-i(t-s)A}F(s)ds\|_{L^{a}H^{1}_{b}}\lesssim
    \|F\|_{L^{\widetilde{a'}}H^{1}_{\widetilde{b}'}}.
  \end{equation}
\end{proposition}

As usual, the $L^{a}L^{b}$ norms in
\eqref{eq:strichA}, \eqref{eq:strichAder}
can be restricted to $L^{a}_{I}L^{b}$ norms 
on an arbitrary time interval $I$ containing 0.

\section{The virial identity}\label{sec:6}

In Sections \ref{sec:2} to \ref{sec:4}
we have proved the existence of a
critical solution, and the compactness of its flow,
for an arbitrary operator $A$ satisfying the abstract
assumptions \textbf{A1}--\textbf{A6} and \textbf{B1}--\textbf{B3},
under the hypothesis $0<\mathcal{E}_{crit}<\infty$.
We can repackage this part of the argument as follows:
suppose a non--scattering solution exists, which is equivalent
to $0<\mathcal{E}_{crit}<\infty$. Then we can construct
a solution with energy $\mathcal{E}_{crit}$ whose flow is
relatively compact in $H^{1}$.
We shall now check that the operator \eqref{eq:definA} satisfies
the abstract assumptions
\textbf{A1}--\textbf{A6}, \textbf{B1}--\textbf{B3},
and this will imply Theorem \ref{the:critA}
as a consequence.

We apply to the equation a gauge transform
\begin{equation}\label{eq:gauge}
  \textstyle
  w(t,x)=e^{-iB(x)}u(t,x),
  \qquad
  B(x)=\int_{0}^{x}b(y)dy
\end{equation}
which gives the reductions
\begin{equation*}
  w_{x}=e^{-iB(x)}\partial_{b}u,
  \qquad
  (aw_{x})_{x}=
  e^{-iB(x)}\partial_{b}(a \partial_{b}u),
  \qquad
  c(x)w=e^{-iB(x)}c(x)u
\end{equation*}
and $F(w)=e^{-iB(x)}F(u)$, where $F(z)=|z|^{\beta-1}z$.
Hence, multiplying the equation by $e^{-iB(x)}$ and
working with $w=e^{-iB(x)}u$, we see that $w$ satisfies the
equation
\begin{equation*}
  iw_{t} +(a(x)\partial_{x}u)_{x} - c(x)w=F(w).
\end{equation*}
Note that $w$ scatters if and only if $u$ scatters.
This shows that we can assume $b \equiv0$. 
From now on, we take $b \equiv0$ and write $u$ instead of $w$.

\begin{proposition}\label{pro:checkA}
  If the coefficients $a,b,c$ satisfy \eqref{eq:assabc}
  then the operator $A$ satisfies Assumptions
  \textbf{A1}--\textbf{A6} and \textbf{B1}--\textbf{B3},
  with the choice $A_{\infty}=-\partial_{x}^{2}$.
\end{proposition}

\begin{proof}
  We know already that under \eqref{eq:assabc} the operator
  $A$ is selfadjoint, non negative, with
  $D(A)=H^{2}(\mathbb{R})$. In addition, the assumptions
  of Proposition \ref{pro:strichA} are satisfied, hence
  both the dispersive estimate \eqref{eq:dispA} and
  the Strichartz estimates \eqref{eq:strichA}, 
  \eqref{eq:strichAder} are valid.

  (\textbf{A1}).
  We have easily
  \begin{equation*}
    (Av,v)\le
    \|a\|_{L^{\infty}}\|\partial v\|_{L^{2}}^{2}+
    \|c\|_{L^{1}}\|v\|_{L^{\infty}}^{2}
    \lesssim \|v\|_{H^{1}}^{2}
  \end{equation*}
  while, using $a\ge a_{0}>0$ and $c\ge0$,
  \begin{equation*}
    (Av,v)\ge a_{0}\|\partial v\|_{L^{2}}^{2}
  \end{equation*}
  from which it follows that
  \begin{equation*}
    (Av, v)_{L^2} + \|v\|_{L^2}^2 \simeq \|v\|_{H^1}^2.
  \end{equation*}
  Since
  \begin{equation*}
    (A_zv, v)_{L^2} = (A \tau_z v, \tau_z v)
  \end{equation*}
  by the translational invariance of the $L^2, H^1$ norms
  we conclude
  \begin{equation*}
    (A_zv, v)_{L^2} + \|v\|^2_{L^2} \simeq \|v\|_{H^1}^2.
  \end{equation*}
  
  (\textbf{A2}). 
  This follows by the standard properties of the unitary group:
  \begin{equation*}
    \|e^{-itA}\psi\|_{H^1}^2 \simeq
    (A e^{-itA} \psi ,e^{-itA} \psi)_{L^2} 
    + \|e^{-itA} \psi \|_{L^2}^{2}
    \lesssim \| \psi \|_{H^1}^{2}.
  \end{equation*}
  
  (\textbf{A3}). 
  Note that
  \begin{equation*}
    A_{z}\psi=-(a(x-z)\partial_{x} \psi)_{x}+c(x-z)\psi.
  \end{equation*}

  If $(x_{k})$ is bounded, up to a subsequence,
  we can assume that it converges to a real number;
  it is not restrictive to assume that $x_{k}\to0$.
  Then we can write
  \begin{equation*}
    \|(c(x-x_{k})-c(x))\psi(x)\|_{L^{1}}\le
    \|\tau_{x_{k}}c-c\|_{L^{1}}\|\psi\|_{L^{\infty}}\lesssim
    \|\tau_{x_{k}}c-c\|_{L^{1}}\|\psi\|_{H^{1}}.
  \end{equation*}
  Since $c\in L^{1}$ we have $\tau_{x_{k}}c\to c$ in $L^{1}$
  as $k\to \infty$ implying
  \begin{equation*}
    c(x-x_{k})\psi\to c(x)\psi
    \quad\text{in}\quad L^{1}
    \qquad\text{and hence in}\quad H^{-1}. 
  \end{equation*}
  Moreover we can write
  \begin{equation*}
    \|(a(x-x_{k})-a(x))\psi_{x}\|_{L^{2}}\le
    \|\tau_{x_{k}}a-a\|_{L^{\infty}}\|\psi\|_{H^{1}}.
  \end{equation*}
  From $a'\in L^{1}$ it follows that $a$ is uniformly
  continuous, implying $\tau_{x_{k}}a\to a$ uniformly on
  $\mathbb{R}$. We conclude that
  \begin{equation*}
    a(x-x_{k})\partial_{x}\psi\to a(x)\partial_{x}\psi
    \quad\text{in}\quad L^{2}
    \quad \Rightarrow \quad
    (a(x-x_{k})\partial_{x}\psi)_{x}\to (a(x)\partial_{x}\psi)_{x}
    \quad\text{in}\quad H^{-1}.
  \end{equation*}
  Summing up, we have proved 
  $A_{x_{k}}\psi\to A \psi$ in $H^{-1}$.

  If $(x_{k})$ is unbounded, up to a subsequence, we can assume
  that $x_{k}\to+\infty$ (the case of $-\infty$ is similar).
  If $\psi\in C_{c}(\mathbb{R})$ we have obviously
  $c(x-x_{k})\psi\to0$ in $L^{1}$ as $k\to \infty$
  since $c\in L^{1}$,
  and by density it follows that $c(x-x_{k})\psi\to0$ in $L^{1}$
  for arbitrary $\psi\in H^{1}$; this implies
  \begin{equation*}
    c(x-x_{k})\psi\to 0
    \quad\text{in}\quad H^{-1}.
  \end{equation*}
  By a similar argument, $a(x-x_{k})\psi_{x}\to \psi_{x}$
  in $L^{2}$ provided $a(x)\to 1$ as $x\to\pm \infty$, and this
  implies
  \begin{equation*}
    (a(x-x_{k})\partial_{x}\psi)_{x}\to \psi_{xx}
    \quad\text{in}\quad H^{-1}.
  \end{equation*}
  Combining the two properties we obtain
  $A_{x_{k}}\psi\to \psi_{xx}$ in $H^{-1}$.

  (\textbf{A4}). 
  Again, it is sufficient to consider the two
  cases $x_{k}\to \bar{x}$, or $x_{k}\to+\infty$.
  In the case $x_{k}\to \bar{x}$, we claim that
  $e^{-is_k A_{x_k}}\psi \to e^{-i\bar s A_{\bar x}}\psi$ in $H^1$. 
  Indeed, since 
  $e^{-is_k A_{x_k}}\psi = \tau_{-x_k}e^{-is_k A}\tau_{x_k}\psi$ 
  and similarly for $e^{-i\bar s A_{\bar x}} \psi$, 
  by continuity of the translation operator we must only
  show that 
  $e^{-is_k A}\tau_{x_k}\psi \to e^{-i\bar s A}\tau_{\bar x}\psi$ 
  in $H^1$. We decompose \begin{align*}
  \textstyle
  e^{-is_k A}\tau_{x_k}\psi - e^{-i\bar s A}\tau_{\bar x}\psi &= (e^{-is_k A}\tau_{x_k}\psi - e^{-i s_k A}\tau_{\bar x}\psi) +( e^{-i s_k A}\tau_{\bar x}\psi - e^{-i\bar s A}\tau_{\bar x}\psi).
  \end{align*}
  For the first term, by \textbf{A2} we have
  \begin{equation*}
    \| e^{-is_k A}\tau_{x_k}\psi - e^{-i s_k A}
      \tau_{\bar x}\psi \|_{H^1} \lesssim 
    \| \tau_{x_k}\psi - \tau_{\bar x}\psi \|_{H^1} \to 0,
  \end{equation*}
  while for the second term we have,
  by the continuity of the flow map,
  \begin{equation*}
    \| e^{-i s_k A}\tau_{\bar x}\psi - e^{-i\bar s A}
    \tau_{\bar x}\psi \|_{H^1} \to 0.
  \end{equation*}

  In the second case $x_k \to + \infty$, we shall prove that
  $e^{-is_k A_{x_k}}\psi \to e^{i\bar s \partial^{2}_x}\psi$ in $H^1$.  
  Indeed, thanks to \textbf{A2}, it is sufficient to 
  consider $\psi \in C^\infty_0$ by a density argument.
  We can write
  \begin{equation*}
  \begin{split}
    \| e^{-is_k A_{x_k}}\psi &- e^{i\bar s \partial^{2}_{x}}\psi \|_{H^1}
    =
    \| e^{-is_k A} \tau_{x_k} \psi - e^{i\bar s \partial^{2}_{x}}
        \tau_{x_k}\psi \|_{H^1}
    \\
    \le&
    \| e^{-is_k A} \tau_{x_k} \psi - 
      e^{i s_{x_k} \partial^{2}_{x}}\tau_{x_k}\psi \|_{H^1} +
    \| e^{i s_{x_k} \partial^{2}_{x}}\tau_{x_k}\psi - 
    e^{i\bar s \partial^{2}_{x}}\tau_{x_k}\psi \|_{H^1}.
  \end{split}
  \end{equation*}
  The second term tends to $0$ by the continuity of the flow 
  $e^{it\partial^{2}_{x}}$; it remains to show that
  \begin{equation*}
  \textstyle
  \| e^{-is_k A} \tau_{x_k} \psi - 
    e^{i s_{x_k} \partial^{2}}\tau_{x_k}\psi \|_{H^1} \to 0.
  \end{equation*}

  The function 
  $u_{k}(t) = -( e^{-it A} \tau_{x_k} \psi 
    - e^{-i t \partial^{2}_{x}}\tau_{x_k}\psi )$ satisfies
  \begin{equation*}
    i\partial_tu_{k} + Au_{k} 
    = (A + \partial^{2}_{x})e^{-it\partial^{2}_{x}} \tau_{x_k}\psi,
    \qquad u_{k}(0) = 0.
  \end{equation*}
  We shall write $v(t,x)=e^{-it \partial^{2}_{x}}\psi$.
  Note that
  \begin{equation*}
    (A+\partial^{2}_{x})\tau_{x_k}v=(1-a) \tau_{x_k}v_{xx}
    -a_{x} \tau_{x_k}v_{x}+c \tau_{x_k}v=
    \tau_{x_{k}}R_{k}(t,x)
  \end{equation*}
  where
  \begin{equation*}
    R_{k}(t,x)=(1-a(x+x_{k}))v_{xx}-a_{x}(x+x_{k})v_{x}
    +c(x+x_{k})v.
  \end{equation*}
  Applying a Strichartz estimate \eqref{eq:strichAder} 
  on the time interval $I=[0,\bar{s}+1]$ we get
  \begin{equation}\label{eq:eninhom2}
    \|u_{k}\|_{L^{\infty}_{I}H^{1}}\lesssim
    \|R_{k}(t,x)\|_{L^{4/3}_{I}W^{1,1}}.
  \end{equation}
  Recall that if $\psi\in \mathscr{S}$ then 
  $v=e^{it \Delta}\psi\in \mathscr{S}$ for all times, and
  actually, for $t$ in a bounded interval,
  $e^{it \Delta}\psi$ remains in a bounded subset
  of $\mathscr{S}$. This follows easily by taking Fourier
  transform and noticing that multiplication
  by $e^{-it|\xi|^{2}}$ enjoys the same property.
  In particular, we have for arbitrary $N>0$
  \begin{equation*}
    |v(t,x)|+|v_{x}|+|v_{xx}|+|v_{xxx}|\le C_{N}\bra{x}^{-N},
    \qquad
    t\in I,\quad x\in \mathbb{R}.
  \end{equation*}
  Using this estimate, the term $R(t,x)$ can be estimated as
  \begin{equation}\label{eq:estbigR}
    \|R_{k}(t,x)\|_{L^{4/3}_{I}W^{1,1}}\lesssim
    \|\bra{x}^{-N}\tau_{-x_{k}}
    (|a-1|+|a_{x}|+|a_{xx}|+c+|c_{x}|)
      \|_{L^{1}(\mathbb{R})}.
  \end{equation}
  Since $a-1\to0$ as $x\to\pm \infty$ and
  $a_{x},a_{xx},c,c_{x}\in L^{1}(\mathbb{R})$, we obtain
  easily that the right hand side of \eqref{eq:estbigR} tends
  to 0 as $x_{k}\to\pm \infty$.
  
  (\textbf{A5}, \textbf{A6}). 
  If $(s_{k})$ is bounded, \textbf{A5} is a consequence of
  \textbf{A4}. If $(s_{k})$ is unbounded, passing to
  a subsequence it is sufficient to consider the case
  $s_{k}\to+\infty$. Then the dispersive estimates in 
  Proposition \ref{pro:strichA} implies
  \begin{equation*}
    \textstyle
    \|e^{-is_k A_{x_k}} \psi\|_{L^r} \lesssim 
    |s_k|^{\frac{1}{r} - \frac{1}{2}}\|\psi\|_{L^{r'}} \to 0
  \end{equation*}
  which is stronger than \textbf{A5}.
  The proof of \textbf{A6} is similar.
  
  (\textbf{B1}).
  The validity of Strichartz estimates is ensured by
  Proposition \ref{pro:strichA}
  
  (\textbf{B2}). Mass and energy conservation are proved
  via the standard identities
  \begin{align*}
    \textstyle
    \partial_t M(t) = 2 \Re(i \int u(A \bar u +
      |u|^{\beta -1} \bar u)) = 0,
  \end{align*}
  \begin{align*}
    \textstyle
    \partial_tE(t) = 2\Re(i\int (Au + |u|^{\beta - 1} u)
      \overline{(Au + |u|^{\beta -1}u)}) = 0.
  \end{align*}
  
  (\textbf{B3}). 
  We choose $A_\infty = - \partial^{2}_{x}$.  
  We start from equation (\ref{eq:asystr}), with 
  $x_n \to + \infty$. 
  By density we can assume $\psi \in C^\infty_0$. 
  Fix $\epsilon>0$. The dispersive estimate 
  \eqref{eq:dispA} implies
  \begin{equation*}
    \|e^{-itA} \tau_{x_n} \psi \|_{L^r} 
    \lesssim |t|^{\frac 1r-\frac 12}
    \| \psi \|_{L^{r'}}
  \end{equation*}
  uniformly in $x_{n}$. Since 
  $p\left(\frac 12-\frac1r\right) = 2$,
  the right hand side belongs to $L^p_t(1,\infty)$,
  and we can find $T_{\epsilon}>0$ such that
  \begin{equation*}
    \|e^{-itA} \tau_{x_n} \psi \|
    _{L^{p}([T_{\epsilon},\infty);L^r)}
    \le \epsilon/2
  \end{equation*}
  and a similar computation is valid for the flow
  $e^{it \partial^{2}_{x}}$. This gives for all $x_{n}$
  \begin{equation*}
    \|(e^{-itA}-e^{it \partial^{2}_{x}})\tau_{x_n} \psi
    \|_{L^{p}([T_{\epsilon},\infty); L^r)}
    \le \epsilon.
  \end{equation*}
  Thus, writing
  \begin{equation*}
    L^{p}_{T}L^{q}=L^{p}([0,T];L^{q}(\mathbb{R}))
  \end{equation*}
  it remains to show that for a fixed $T$
  \begin{equation*}
  \textstyle
  \| (e^{it \partial^{2}_{x}} -e^{-itA}) \tau_{x_n} \psi \|
    _{L^p_{T}L^r} \to 0 \quad\text{ as }\quad n \to \infty.
  \end{equation*}
  The function 
  $\varphi = (e^{it \partial^{2}_{x}} -e^{-itA}) \tau_{x_n} \psi$ 
  solves the problem
  \begin{equation*}
  \textstyle
    i\partial_t \varphi + A \varphi = 
    (A +\partial^{2}_{x}) e^{it\partial^{2}_{x}}\tau_{x_n} \psi,
    \qquad\varphi(0) = 0.
  \end{equation*}
  By Strichartz estimates (\ref{eq:strichA}) we have
  \begin{equation*}
    \| (e^{it\partial^{2}_{x}} -e^{-itA}) \tau_{x_n} \psi \|
     _{L^p_{T}L^r}\lesssim
    \| (A +\partial^{2}_{x}) e^{it\partial^{2}_{x}}\tau_{x_n} \psi\|
      _{L^{\frac{4}{3}}_{T}L^1}\lesssim
     T^{\frac{3}{4}} 
         \| (A + \partial^{2}_{x}) e^{it\partial^{2}_{x}}\tau_{x_n} \psi\|
         _{L^\infty_{T}L^1},
  \end{equation*}
  which can be treated as \eqref{eq:estbigR} above.

  The proof of \eqref{eq:asystrin} is similar.
  By density, we can take $F\in \mathscr{S}(\mathbb{R}^{n+1})$.
  Fix $\epsilon>0$. 
  By Strichartz estimates we have
  (with $L^{p}L^{r}=L^{p}(\mathbb{R}^{+};L^{r})$)
  \begin{equation*}
    \textstyle
    \|\int_{-\infty}^{-T}e^{-i(t-s)A}\tau_{x_{n}}Fds\|
    _{L^{p}L^{r}}
    \lesssim
    \|F\|_{L^{p'}((-\infty,-T];L^{r'})}\le \epsilon
  \end{equation*}
  for $T>0$ large enough, and a similar estimate holds for
  $e^{it \partial^{2}_{x}}$. To estimate the piece
  $\int_{-T}^{t}$, we use the dispersive estimate and write
  \begin{equation*}
    \textstyle
    \|\int_{-T}^{t}e^{-i(t-s)A}\tau_{x_{n}}Fds\|_{L^{r}_{x}}
    \le
    \int\|e^{-i(t-s)A}\tau_{x_{n}}F\|_{L^{r}}ds
    \lesssim
    \int|t-s|^{\frac 1r-\frac 12}\|F\|_{L^{r'}_{x}}ds.
  \end{equation*}
  By the Hardy--Sobolev inequality,
  $g(t)=|t|^{\frac 1r-\frac 12}*\|F\|_{L^{r'}}$
  belongs to $L^{p}(\mathbb{R})$, hence
  \begin{equation*}
    \textstyle
    \|\int_{-T}^{t}e^{-i(t-s)A}\tau_{x_{n}}Fds\|
    _{L^{p}([T,\infty));L^{r}_{x})}
    \lesssim
    \|g\|_{L^{p}([T,\infty))}\le \epsilon
  \end{equation*}
  provided $T$ is large enough;
  a similar estimate holds for the flow $e^{it\partial^{2}_{x}}$.
  It remains to estimate, for a fixed $T>0$ and
  for $t\in[0,T]$, the integral
  \begin{equation*}
    \textstyle
    \int_{-T}^{t}(e^{-i(t-s)A}-e^{i(t-s)\partial^{2}_{x}})
    \tau_{x_{n}}F(s)ds.
  \end{equation*}
  Changing variables $s'=s+T$, $t'=t+T$ this becomes
  \begin{equation*}
    \textstyle
    =\int_{0}^{t}(e^{-i(t-s)A}-e^{i(t-s)\partial^{2}_{x}})
    \tau_{x_{n}}\widetilde{F}(s)ds,
    \qquad
    \widetilde{F}(s)=F(s-T)
  \end{equation*}
  for $t\in[T,2T]$.
  The function 
  $\phi = \int_{0}^t( e^{-i(t-s)\partial^{2}_{x}} - e^{-i(t-s)A}) 
    \tau_{x_n}\widetilde{F} ds$ 
  solves the problem
  \begin{equation*}
    \textstyle
    i \partial_t \phi + A \phi = 
    (A +\partial^{2}_{x}) 
    \int_{0}^t e^{-i(t-s)\partial_{x}^{2}} \tau_{x_n}\widetilde{F} ds, 
    \qquad
    \phi(0) = 0.
  \end{equation*}
  By Strichartz estimates (\ref{eq:strichA}) we have then
  \begin{equation*}
  \begin{split}
    \textstyle
    \|\phi\|_{L^{p}([0,2T];L^{r})}\le
    & \textstyle
    \|(A+\partial^{2}_{x})\int_{0}^{t}e^{-i(t-s)\partial_{x}^{2}}
      \tau_{x_{n}}\widetilde{F}(s)ds\|
    _{L^{4/3}([0,2T];L^{1})}
    \\
    \le &
    T^{3/4}
    \|(A+\partial^{2})\tau_{x_{n}}\int_{0}^{t}e^{-i(t-s)\partial_{x}^{2}}
      \tau_{x_{n}}\widetilde{F}(s)ds\|
    _{L^{\infty}([0,2T];L^{1})}
    \end{split}
  \end{equation*}
  which can be treated as \eqref{eq:estbigR} since 
  $\int_{0}^t e^{-i(t-s)\partial_{x}^{2}} 
    \widetilde{F} ds  \in C([0,2T], H^1)$.
\end{proof}

This completes the proof ot Theorem \ref{the:critA}.

\begin{remark}\label{rem:disp}
  Examining the proof of Proposition \ref{pro:checkA},
  one checks that all properties can be derived from the
  dispersive estimate \eqref{eq:dispA}.
  Thus the same scheme works for other 1D operators $A$
  for which a dispersive estimate is available.
  This includes the Laplacian with a delta potential,
  studied in \cite{BanicaVisciglia16-a}, and the
  Schr\"{o}dinger operator with a potential studied in
  \cite{Lafontaine16-a}.
\end{remark}

Consider now Theorem \ref{the:one}.
We need to show that the critical 
solution constructed in Section \ref{sec:4}
can not possibly exist, due to the compactness of the flow
which violates the dispersive properties of the equation.
The proof will be based on a virial identity which
requires the stronger set of assumptions \eqref{eq:assabcbis}.
Thus in the following we assume that the coefficients
$a,b,c$ of the operator \eqref{eq:definA} satisfy
both \eqref{eq:assabc} and \eqref{eq:assabcbis}.

Given a smooth, real valued
and bounded function $m(x)$,
consider the quantity
\begin{equation*}
  \theta(t)=\int_{\mathbb{R}} m(x)|u|^{2}dx=(mu,u)_{L^{2}}
\end{equation*}
where $u$ is a solution of the equation
\begin{equation}\label{eq:eqH}
  iu_{t}=Lu+c(x)u+F(u),\qquad
  Lu=-(a u_{x})_{x},
  \qquad
  F(u)=|u|^{\beta-1}u.
\end{equation}
Differentiating $\theta(t)$ we get formally
\begin{equation}\label{eq:virial0}
  \textstyle
  \theta'(t)=-\Re(i[L,m]u,u)_{L^{2}}
  =2\Re i\int am' u'\overline{u}dx
\end{equation}
where the term $cu$ gives no contribution since $[c,m]=0$.
Computing $\theta''$ we get
\begin{equation*}
  \textstyle
  \theta''(t)=
  ([L,[L,m]]u,u)_{L^{2}}+ 2(a m'c'u,u)_{L^{2}}
  +2\Re([L,m]u,f(u))
\end{equation*}
and hence the virial identity
\begin{equation}\label{eq:thevirial}
  \textstyle
  \theta''(t)=
  \int (L^{2}m+2am'c')|u|^{2}+
  4\int a(Lm+\frac 12a'm')|u_{x}|^{2}+
  2 \frac{\beta-1}{\beta+1}
  \int (Lm)|u|^{\beta+1}.
\end{equation}
To simplify this expression we write
\begin{equation*}
  \gamma(x)=a(x) m'(x)
\end{equation*}
so that $\gamma'=-Lm$, and changing sign we obtain
\begin{equation}\label{eq:virial1}
  \textstyle
  -\theta''(t)=
  \int[-(a \gamma'')'-2 \gamma c']\cdot|u|^{2}+
  2\int(2a \gamma'-a'\gamma)|u_{x}|^{2}+
  \frac{2}{\beta+1}\int(\beta-1)\gamma'|u|^{\beta+1}
\end{equation}
while \eqref{eq:virial0} simplifies to
\begin{equation}\label{eq:virial3}
  \textstyle
  \theta'(t)=-\Re(i[L,m]u,u)_{L^{2}}
  =2\Re i\int \gamma u'\overline{u}dx.
\end{equation}

We now choose $\gamma$ smooth and compactly
supported, so that
the rigorous justification of \eqref{eq:virial1},
\eqref{eq:virial3} reduces to a routine integration by parts.
Precisely, we choose
\begin{equation}\label{eq:choiceg}
  \textstyle
  \gamma(x)=(x-\frac{x^{3}}{6R^{2}})\cdot \chi(\frac xR)
\end{equation}
where $R>1$ is a parameter to be fixed, and 
$\chi\in C_{c}^{\infty}(\mathbb{R})$ is such that
\begin{equation*}
  \textstyle
  \one{B(0,1)}\le \chi\le \one{B(0,2)}.
\end{equation*}
Under the stronger \eqref{eq:assabcbis},
we now estimate the coefficients in \eqref{eq:virial1}
\begin{equation*}
  I=2(2a \gamma'-a'\gamma),
  \qquad
  II=-(a \gamma'')'-2 \gamma c',
  \qquad
  III=(\beta-1)\gamma'.
\end{equation*}

In the region $|x|\le R$ we have, using
the first condition in \eqref{eq:assabcbis},
\begin{equation*}
  \textstyle
  I=
  a(4-\frac{2}{R^{2}}x^{2})-
  2 a'x(1-\frac{x^{2}}{6R^{2}})
  \ge
  a(4-\frac{2}{R^{2}}x^{2})-
  2(\frac{6}{5}-\delta)a(1-\frac{x^{2}}{6R^{2}})
  \ge \frac 5 \delta a
\end{equation*}
while by both conditions in \eqref{eq:assabcbis} we get
\begin{equation*}
  \textstyle
  II=
  \frac{1}{R^{2}}(a'x+a)-2(1-\frac{x^{2}}{6R^{2}})xc'
  \ge 
  \frac{1}{R^{2}}\delta a
\end{equation*}
and finally
\begin{equation*}
  \textstyle
  III=(\beta-1)(1-\frac{x^{2}}{2R^{2}})\ge \frac{\beta-1}{2}.
\end{equation*}
Summing up we have proved
\begin{equation}\label{eq:finalI}
  I\ge C,\qquad
  II\ge \frac{C}{R^{2}},\qquad
  III\ge C
\end{equation}
for some $C>0$ independent of $R$.

Consider next the region $|x|\ge R$; since $\gamma$ is 
supported in $|x|\le 2R$ we may assume
$R\le|x|\le 2R$. Then we have easily
\begin{equation*}
  \textstyle
  |I|\lesssim|a' \gamma|+|a \gamma'|\lesssim
  |a'x|\frac{|\gamma|}{|x|}+|\gamma'|\le C',
\end{equation*}
\begin{equation*}
  \textstyle
  |II|\lesssim|a'x|\frac{|\gamma''|}{|x|}+|\gamma'''|+
  |xc'|
  \le \frac{C'}{R^{2}}
\end{equation*}
\begin{equation*}
  \textstyle
  |III|=(\beta-1)|\gamma'|\le C'
\end{equation*}
and summing up
\begin{equation}\label{eq:final2}
  |I|\le C',\qquad
  |II|\le \frac{C'}{R^{2}},\qquad
  |III|\le C'
\end{equation}
for a suitable constant $C'$ also independent of $R$.

We now go back to the virial identity \eqref{eq:virial1}.
Splitting the integral in the regions $|x|\le R$ and
$|x|>R$ and using \eqref{eq:finalI}, \eqref{eq:final2};
we obtain
\begin{equation*}
  \textstyle
  \int_{|x|\le R}
  \bigl(
  \frac{|u|^{2}}{R^{2}}
  +|u_{x}|^{2}+\frac{1}{\beta+1}|u|^{\beta+1}
  \bigr)dx
  \lesssim
  -\theta''(t)+
  \int_{|x|\ge R}
  \bigl(
  \frac{|u|^{2}}{R^{2}}
  +|u_{x}|^{2}+\frac{2}{\beta+1}|u|^{\beta+1}
  \bigr)dx.
\end{equation*}
By \eqref{eq:comcr1} and \eqref{eq:comcr2},
for any $\sigma\in(0,1)$ we can find $R>0$ such that
\begin{equation}\label{eq:virint}
  \textstyle
  \int_{|x|\ge R}
  \bigl(
  \frac{|u|^{2}}{R^{2}}
  +|u_{x}|^{2}+\frac{2}{\beta+1}|u|^{\beta+1}
  \bigr)dx
  \le 
  \sigma
  \int_{|x|\le R}
  \bigl(
  \frac{|u|^{2}}{R^{2}}
  +|u_{x}|^{2}+\frac{2}{\beta+1}|u|^{\beta+1}
  \bigr)dx.
\end{equation}
Choosing $\sigma$ small enough, we can then absorb
the integral at the right of \eqref{eq:virint} 
in the integral at the left; this gives
\begin{equation*}
  \textstyle
  \int_{|x|\le R}
  \bigl(
  \frac{|u|^{2}}{R^{2}}
  +|u_{x}|^{2}+\frac{2}{\beta+1}|u|^{\beta+1}
  \bigr)dx
  \lesssim
  -\theta''(t)
\end{equation*}
and using again \eqref{eq:virint} we arrive at
\begin{equation}\label{eq:viralm}
  \textstyle
  Z(t):=
  \int_{\mathbb{R}}
  \bigl(
  \frac{|u|^{2}}{R^{2}}
  +|u_{x}|^{2}+\frac{2}{\beta+1}|u|^{\beta+1}
  \bigr)dx
  \lesssim
  -\theta''(t).
\end{equation}
Note that $Z(t)$ is actually constant in time by conservation of
mass and energy.

We next integrate \eqref{eq:viralm} in time from 0 to 
$T>0$; by \eqref{eq:virial3}, the right hand side gives
\begin{equation*}
  \textstyle
  -\int_{0}^{T}\theta''(t)dt=
  -\theta'(t)\vert_{t=0}^{t=T}
  =-2\Re i\int \gamma u'\overline{u}dx\vert_{t=0}^{t=T}
  \lesssim
  Z(t)=Z(0)
\end{equation*}
by conservation of energy, and hence for all $T>0$
\begin{equation}\label{eq:finalvir}
  \textstyle
  \int_{0}^{T} Z(0)dt
  \lesssim
  Z(0)
\end{equation}
Since $Z(0)\simeq \mathcal{E}_{crit}>0$,
we see that \eqref{eq:finalvir} contradicts the assumption 
$\mathcal{E}_{crit}<\infty$.
We conclude that $\mathcal{E}_{crit}=\infty$, 
and this proves Theorem \ref{the:critA}.

\end{document}